\documentclass[11pt,a4paper, twoside]{article}
\usepackage[utf8]{inputenc}
\usepackage[T1]{fontenc}
\usepackage{ucs}
\usepackage{amsmath}
\usepackage{amsfonts}
\usepackage{amssymb}
\usepackage{amsthm}
\usepackage[all]{xy}
\usepackage{enumerate}
\usepackage{enumitem}
\usepackage{tikz}
\usepackage{tikz-cd}
\usepackage{pgfplots}
\usetikzlibrary{patterns}
\usetikzlibrary{matrix,arrows}
\usepackage{ifthen}
\usepackage{everypage}
\usepackage{verbatim} 
\usepackage{mathtools}
\usepackage{url}
\usepackage{mathrsfs}
\usepackage{indentfirst}
\usepackage{todonotes}
\usepackage{hhline}
\usepackage[top=3cm, bottom=3cm, left=3cm, right=3cm]{geometry}
\usepackage{thmtools, thm-restate}
\usepackage{amssymb}
\usepackage{fancyhdr}
\usepackage[title]{appendix}

\usepackage[backref=page, colorlinks=true]{hyperref}

\DeclareMathOperator{\im}{Im}

\DeclareMathOperator{\gal}{Gal}
\DeclareMathOperator{\pic}{Pic}

\DeclareMathOperator{\et}{\acute{e}t}

\DeclareMathOperator{\mor}{Hom}

\DeclareMathOperator{\spec}{Spec}
\DeclareMathOperator{\Res}{Res}
\DeclareMathOperator{\res}{res}
\DeclareMathOperator{\cores}{cores}

\DeclareMathOperator{\coh}{H}

\DeclareMathOperator{\divcart}{div}

\DeclareMathOperator{\br}{Br}
\DeclareMathOperator{\ns}{NS}

\DeclareMathOperator{\cont}{cont}

\DeclareMathOperator{\sep}{s}
\DeclareMathOperator{\cd}{cd}
\DeclareMathOperator{\inv}{inv}

\theoremstyle{definition}

\newtheorem{thm}{Theorem}[section]
\newtheorem{prop}[thm]{Proposition}
\newtheorem{defi}[thm]{Definition}
\newtheorem{lem}[thm]{Lemma}
\newtheorem{cor}[thm]{Corollary}

\newtheorem{rem}[thm]{Remark}
\newtheorem{remi}[thm]{Reminder}

\newtheorem{quest}[thm]{Question}

\newtheorem{conj}[thm]{Conjecture}

\newcommand\restr[2]{{
  \left.\kern-\nulldelimiterspace 
  #1 
  \vphantom{\big|} 
  \right|_{#2} 
  }}

\title{Unirationality and $R$-equivalence for conic bundles over quasi-finite fields}
\author{Elyes Boughattas}
\date{}

\makeatletter
\def\blfootnote{\gdef\@thefnmark{}\@footnotetext}
\makeatother

\AtEndDocument{\bigskip{\footnotesize%
  \textsc{Department of Mathematical Sciences, University of Bath - Claverton Down, Bath, BA2~7AY, United Kingdom} \par  
  \textit{Email address}: \texttt{eb2751@bath.ac.uk}
  }}

\begin{document}

\pagestyle{fancy}
\fancyhead{}
\fancyfoot{}
\fancyhead[CO]{\sc{Conic bundles over quasi-finite fields}}
\fancyhead[CE]{\sc{Elyes Boughattas}}
\fancyhead[RE,LO]{\thepage}
\renewcommand{\headrulewidth}{0pt}

\blfootnote{\textit{Date}: \today}
\maketitle

\begin{abstract}
Yanchevskiĭ had asked whether conic bundle surfaces over $\mathbf{P}^1_k$ are unirational when~$k$ is a finite field. We give a partial answer to his question by showing that for quasi-finite fields $k$ (e.g.\ finite fields) a regular conic bundle $X$ over~$\mathbf{P}^1_k$ is unirational if all non-split fibres lie over rational points. For large finite fields $k$, this beats a previous result of Mestre. Under the same assumption, we also prove that all rational points of $X$ are $R$-equivalent.
\end{abstract}

\section{Introduction}

Let $k$ be a field.
A $k$-variety is said to be \textit{rational} (resp. \textit{unirational})
 if there exist $n\in\mathbf{Z}_{\geq0}$ and a rational map $\mathbf{P}^n_k\dashrightarrow X$ that is birational (resp. dominant).
In the case of smooth geometrically rational projective surfaces, the classification of their minimal models by seminal works of Manin \cite{MR225780}\cite{MR225781} and Iskovskikh~\cite{MR525940} splits the study of unirationality into two families.
%
The first family consists in del Pezzo surfaces $X$, that is, smooth projective surfaces with ample anticanonical bundle $K_X$. To $X$ is associated its degree $d\coloneqq(K_X)^2\in\{1,\dots,9\}$ and works of Segre, Manin and Kollár showed that~$X$ is unirational whenever $X(k)\neq\emptyset$ and $d\geq3$, over any field $k$. When $d=2$, partial results are given by Salgado, Testa and Várilly-Alvarado~\cite{MR3245139} and the case of finite fields has been fully tackled by Festi and van Luijk \cite{MR3455757}. The second family consists in conic bundles over a smooth curve. To define them, we recall that a \textit{conic} over $k$ is the vanishing locus of a quadric in~$\mathbf{P}^2_k$.
\begin{defi}\label{CONIC}
For any scheme $S$, a \textit{conic bundle over $S$} is a proper and flat morphism $f:X\rightarrow S$ where $X$ is integral, with a smooth generic fibre and such that all geometric fibres of $f$ are isomorphic to a conic. The conic bundle~$f$ is called \textit{regular} if $X$ is regular.
\end{defi}
By Lüroth's theorem, if $X\rightarrow C$ is a conic bundle over a smooth, integral, projective curve $C$, a necessary condition for the unirationality of~$X$ is that~$C\simeq\mathbf{P}^1_k$, which lets us, from now on, restrict our attention to that case. The following question was raised by Iskovskikh in~\cite[\S4.5]{MR220734} (see also \cite[Problem]{MR1211025}).
\begin{quest}\label{problem1}
If $k$ is a field and $X\rightarrow\mathbf{P}^1_k$ is a conic bundle such that $X(k)\neq\emptyset$, is $X$ unirational?
\end{quest}
Although this question is still open, it is expected to have a negative answer in general. Using the terminology of \cite[Definition 0.1]{MR1408492}, we say that a conic is \textit{non-split} if it is singular and irreducible. For a conic bundle $f:X\rightarrow\mathbf{P}^1_k$, we denote by $\delta\in\mathbf{Z}_{\geq0}$ the degree of $\{t\in\mathbf{P}^1_k:X_t\text{ is non-split}\}$, which may also be understood as the number of geometric singular fibres of a minimal model of the generic fibre of $f$. When $k$ has characteristic different from $2$, conic bundles with a $k$-point are $k$-unirational if $\delta\leq7$, by works of Segre, Iskovskikh, Manin, Kollár and Mella (see section~$1$ of \cite{MR3689320} for a summary). Iskovskikh~\cite{MR220734} gave a positive answer to Question \ref{problem1} for $k=\mathbf{R}$, or more generally when~$k$ is a real closed field. Later, Yanchevskiĭ \cite{MR826395} proved the case of Henselian fields, Voronovich that of pseudo-algebraically closed fields~\cite{MR867120} and Yanchevskiĭ generalised this result to pseudo-closed fields, see \cite[Theorem 1]{MR1211025}. Over finite fields, Yanchevskiĭ raised the following question (see the discussion preceding Problem 3 in \cite{MR1171295}).

\begin{quest}[Yanchevskiĭ]\label{yanchoriginal}
Let $\mathbf{F}$ be a finite field. For all conic bundles $f:X\rightarrow\mathbf{P}^1_{\mathbf{F}}$, is~$X$ unirational?
\end{quest}

The only result in this direction was given by Mestre in \cite{MR1381777} who proved that~$X$ is
unirational when $\mathbf{F}$ is finite of characteristic different from $2$ and $|\mathbf{F}|\geq \delta^2\times2^{\delta-3}$. Our first result is a partial answer to Question \ref{yanchoriginal} over the larger family of \textit{$2$-quasi-finite} fields that we first define, jointly with the classical notion of quasi-finite fields. We recall that for a prime number $p$, the ring of $p$-adic integers is denoted by $\mathbf{Z}_p$.

\begin{defi}\label{defi2QF}
Let $k$ be a field. We say that $k$ is \textit{$2$-quasi-finite} if it is perfect and if there exists a set of prime numbers $S$ such that $2\in S$ and the absolute Galois group of $k$ is isomorphic to $\prod_{p\in S}\mathbf{Z}_p$. Furthermore, $k$ is called \textit{quasi-finite} if $S$ is the set of all prime numbers.
\end{defi}
By definition, quasi-finite fields are $2$-quasi-finite
Three classical examples of quasi-finite fields are that of finite fields, Laurent series over an algebraically closed field of characteristic zero and non-principal ultraproducts of finite fields (see e.g.\ \cite[\S7, Proposition~3]{MR229613}).
Our first result may then be stated as follows:
%

\begin{thm}\label{unirationalconicbundles}
Let $k$ be a $2$-quasi-finite field of characteristic different from~$2$ and $f:X\rightarrow\mathbf{P}^1_k$ a regular conic bundle. Denote by~$B$ the reduced divisor of $\mathbf{P}^1_k$ made of those points whose fibre by $f$ is non-split. Consider the following assertion:
\begin{enumerate}[label=($\star$), ref =$\star$]
\item\label{unirationalcondition} \textit{The set $B$ is a union of rational points, one point of degree at most~$2$ and one point whose degree is odd.}
\end{enumerate}
If condition (\ref{unirationalcondition}) is verified, then $X$ is 
unirational.
\end{thm}

The case where $B=\mathbf{P}^1(k)$ is not encapsulated in the aforementioned result of Mestre, whose bound is very restrictive for large $|k|$. The proof relies on a criterion of Enriques for unirationality of conic bundles which consists in building a rational curve on $X$ intersecting a general fibre of~$f$ (see Proposition~\ref{enriques}). For fields $k$ with $\cd(k)\leq1$, we prove that this amounts to constructing a finite morphism $\varphi:\mathbf{P}^1_k\rightarrow\mathbf{P}^1_k$ such that for each $t\in\varphi^{-1}(B)$, we have~$2\mid e(t/\varphi(t))\times[\kappa(t):\kappa(\varphi(t))]$, where $e(t/\varphi(t))$ denotes the ramification index of~$\varphi$ at~$t$ (see Theorem \ref{uniratcrit}). More precisely, we show that the unirationality of~$X$ is equivalent to the existence of such a cover. When~$k$ is $2$-quasi-finite, the latter is achieved under assumption (\ref{unirationalcondition}) by taking~$\varphi$ as a tower of well chosen degree $2$ covers (see Section~\ref{proofunirationalconicbundles}). Eventually, this gives a further insight into the limits of the method used by Mestre, whose construction of $\varphi$ is made of precisely \textit{one} degree $2$ cover of~$\mathbf{P}^1_k$, and a degree $2$ cover satisfying the above property does not exist if~$B$ contains $\mathbf{P}^1(k)$.\\


When studying rational points on varieties, another topic of interest is that of \textit{$R$-equivalence} introduced by Manin in \cite[Chapter II, \S14]{MR833513} and defined as follows.

\begin{defi}
Let $k$ be a field and $X$ a $k$-scheme. Two points $x$ and $y$ in~$X(k)$ are called \textit{directly $R$-equivalent} if there exists a rational map $g:\mathbf{P}^1_k\dashrightarrow X$ such that $x$ and~$y$ both belong to $g(\mathbf{P}^1(k))$. The equivalence relation spanned by this relation is called \textit{$R$-equivalence}. We further say that $R$-equivalence is \textit{trivial} if~$X(k)/R$ has cardinality $1$.
\end{defi}

The $R$-equivalence relation on $X(k)$ turns out to be interesting when, geometrically,~$X$ contains many rational curves. This is the case in the setting of geometrically rational surfaces and more generally of separably rationally connected varieties \cite[Chapter~IV, Definition 3.2]{MR1440180} which are defined as follows.

\begin{defi}
A variety $X$ over a field $k$ is \textit{separably rationally connected} if for any algebraic closure $\overline{k}$ of $k$, there exists an integral $\overline{k}$-variety $T$ and a rational map $e:\mathbf{P}^1_{\overline{k}}\times_{\overline{k}}T\dashrightarrow X_{\overline{k}}$ such that the map $\mathbf{P}^1_{\overline{k}}\times\mathbf{P}^1_{\overline{k}}\times T\dashrightarrow X_{\overline{k}}\times X_{\overline{k}}$ defined by $(z,z',t)\mapsto (e(z,t),e(z',t))$ is dominant and smooth at the generic point.
\end{defi}

Various authors have studied triviality and finiteness of~$X(k)/R$ for several classes of separably rationally connected varieties. For instance, it is already known that regular conic bundles over $\mathbf{P}^1_k$ with at most~$5$ reducible geometric fibres have only one $R$-equivalence class if $k$ is an infinite, perfect, $C_1$ field of characteristic different from~$2$ \cite{MR899411}\cite{MR3434268}. We refer the reader to \cite{MR1715330}, \cite{MR2957693} and \cite{MR3434268} for further results on different families of varieties. Our second result in this paper is on the triviality of $R$-equivalence for conic bundles $X\rightarrow\mathbf{P}^1_k$ over a $2$-quasi-finite field $k$.

\begin{thm}\label{Requivalenceconicbundles}
Let us use the same notations as in Theorem \ref{unirationalconicbundles} and consider the following assertion:
\begin{enumerate}[label=($\star\star$), ref =$\star\star$]
\item\label{Requivalencecondition} \textit{The set $B$ is a union of rational points and one point which has either degree~$2$ or odd degree.}
\end{enumerate}
If condition (\ref{Requivalencecondition}) is verified, then $X(k)/R$ is trivial.
\end{thm}

Eventually, in Corollary~\ref{fromconj2toconj1}, we prove that for a finite field $\mathbf{F}$ and a conic bundle $X\rightarrow\mathbf{P}^1_{\mathbf{F}}$, the variety $X$ is unirational and $R$-equivalence is trivial on it, if one assumes that rational points of a smooth, geometrically integral and separably rationally connected projective surface over $\mathbf{F}(t)$ are dense in its Brauer-Manin set (see Conjecture \ref{BMcarp}). This is an analogue of a conjecture of Colliot-Thélène and Sansuc over number fields \cite{MR605344} that has been studied by various authors. Combined to Theorems~\ref{unirationalconicbundles} and~\ref{Requivalenceconicbundles}, this leads to the following generalisation of Question \ref{yanchoriginal} that we formulate over quasi-finite fields.

\begin{quest}\label{conjyanch}
Let $k$ be a quasi-finite field and $X\rightarrow\mathbf{P}^1_k$ a conic bundle. Is $X$ unirational and $R$-equivalence trivial on $X$?
\end{quest}

\subsection*{Outline of the paper}

In Section \ref{secPreliminaries}, we start by setting notations and definitions in \S\ref{subsecNOT}. In \S\ref{subsec2QF}, we state properties of 2-quasi-finite fields that are used all along the article. The content of \S\ref{subsecCB} is known to the experts. We explain how one can detect split fibres of a regular conic bundle through the residues of its generic fibre. Eventually, in \S\ref{rappelapproximation} we recall the definitions of weak approximation and the Brauer-Manin set over the function field of a curve over a finite field. We also recall the dictionary between rational points of a variety and sections of its models.

Section \ref{secCRIT} is dedicated to a unirationality criterion for conic bundles over fields of cohomological dimension at moste one, stated in Theorem~\ref{uniratcrit}. In~\S\ref{subsecENRIQUES}, we recall a criterion of Enriques on unirationality of conic bundles and we restate it over fields of cohomological dimension at most one, in Corollary \ref{enriquesP1}. Eventually, a proof of Theorem \ref{uniratcrit} is given in~\S\ref{subsecUNIRATCRIT}. As a corollary, we also get in~\S\ref{subsecREQUIVCRIT} a criterion for the triviality of $R$-equivalence.

In Section \ref{secMAINRESU}, we apply the criteria given in Section \ref{secCRIT} to prove the unirationality and $R$-equivalence results for conic bundles over a $2$-quasi-finite field stated in Theorems \ref{unirationalconicbundles} and \ref{Requivalenceconicbundles}. This amounts to constructing well-chosen ramified covers $\mathbf{P}^1\rightarrow\mathbf{P}^1$, for which we give general recipes in \S\ref{subsecCOVER}. We prove Theorem \ref{unirationalconicbundles} in~\S\ref{proofunirationalconicbundles} and Theorem~\ref{Requivalenceconicbundles} in~\S\ref{proofRequivalenceconicbundles}.

Eventually, in Section \ref{sectionBMcarp}, we prove in Corollary \ref{fromconj2toconj1} that the unirationality of conic bundle surfaces over finite fields is implied by Conjecture \ref{BMcarp}, which predicts the closure of rational points in the adelic set of a geometrically integral, smooth and separably rationally connected projective surface over $\mathbf{F}(t)$, for any finite field $\mathbf{F}$. The proof goes through Theorem \ref{thmKOLLARPOINTS}, where we show that if rational points are dense in the Brauer-Manin set of a conic bundle surface, then it has weak weak approximation (resp. weak approximation if $|\mathbf{F}|$ is odd) and we conclude by the dictionary of \S\ref{rappelapproximation}.

Appendix \ref{appendixA} is dedicated to the study of Brauer groups of separably rationally connected surfaces. In \S\ref{finiteBRAUER} we supply, in Proposition \ref{brauerfini}, a proof of the finiteness of the algebraic Brauer group, up to constant classes, of a separably rationally connected variety. In \S\ref{BrauerCB}, we show, in Proposition \ref{2primaryCB}, that the Brauer group of a conic bundle surface, up to constant classes, is a $2$-torsion group.

\subsection*{Acknowledgement}
I am very grateful to Olivier Wittenberg who introduced me to this question, for many useful discussions and the time that he generously gave me. I would like to express my gratitude to Antoine Chambert-Loir and Alexei Skorobogatov for their suggestions and detailed comments on a previous version of this article. I also want to thank Jean-Louis Colliot-Thélène for discussions and suggestions on Section \ref{sectionBMcarp}.
The author was supported by a «Contrat doctoral spécifique normalien» from the École normale supérieure de Paris and by UKRI Future Leaders Fellowship~\texttt{MR/V021362/1}.

\section{Preliminaries}\label{secPreliminaries}

\subsection{Notations}\label{subsecNOT}

If $A$ is an abelian group and $n\in\mathbf{Z}_{\geq0}$, we denote by $A\left[n\right]$ the $n$-torsion part of $A$, by~$A\{p\}$ the $p$-primary torsion part of $A$ and $A\{p'\}\coloneqq\bigoplus_{l\text{ prime, }l\neq p}A\{l\}$.

If $X$ is a scheme, we denote by $\br(X)\coloneqq\coh_{\et}^2(X,\mathbf{G}_m)$ the \textit{Brauer group} of $X$ and if $R$ is a commutative ring we set $\br(R)\coloneqq\br(\spec(R))$. If $k$ is a field, $\br(k)$ is also the set of central simple algebras over $k$ up to Brauer equivalence, see \cite[\S2.4]{MR2266528}.

When $R$ is a discrete valuation ring, with function field $K$ and residue field~$\kappa$ whose~characteristic is different from $2$, then we have Serre's \textit{residue} map \cite[Definition 1.4.3.(ii)]{MR4304038} $\partial:\br(K)\{2\}\longrightarrow\coh^1(\kappa,\mathbf{Q}_2/\mathbf{Z}_2)$ whose restriction to the $2$-torsion of $\br(K)$ gives rise to a map
\begin{equation}\label{residuemap}
r:\br(K)[2]\longrightarrow\coh^1(\kappa,\mathbf{Z}/2\mathbf{Z}).
\end{equation}

If $k$ is a field, we denote by $\cd(k)$ its \textit{cohomological dimension} (see e.g.\ \cite[I.\S3.1]{MR1324577}). If $G$ is the absolute Galois group of $k$ and $A$ an abelian group on which $G$ acts trivially, we denote by $\coh^1(k,A)$ the first Galois cohomology group of $G$ with coefficients in $A$ (see \cite[Chapitre I, \S2]{MR1324577}) which is canonically isomorphic to $\coh^1_{\et}(\spec(k),A)$.

When $Y\rightarrow X$ is a morphism of schemes and $A$ is an abelian group, we denote by $\res_{Y/X}:\coh^1_{\et}(X,A)\rightarrow\coh^1_{\et}(Y,A)$ the \textit{restriction} morphism, see e.g.\ \cite[\S2.2.4.(2.7)]{MR4304038}. If $F/E$ is an extension of fields, the restriction map $\res_{\spec(F)/\spec(E)}$ is also denoted $\res_{F/E}$ and we view it as a morphism $\coh^1(E,A)\rightarrow\coh^1(F,A)$ where the absolute Galois groups of~$E$ and $F$ act trivially on $A$. If~$F/E$ is finite and separable, we denote by $\cores_{F/E}:\coh^1(F,A)\rightarrow\coh^1(E,A)$ the \textit{corestriction} map, see e.g.\ \cite[Chapitre I, \S 2.4]{MR1324577}).

A $k$-variety is a separated $k$-scheme of finite type and a \textit{nice curve} over $k$ is a proper, smooth and geometrically integral $k$-variety of dimension one. When $X$ is a scheme, its set of codimension one points is denoted by~$X^{(1)}$. Following \cite[Definition~0.1]{MR1408492}, we say that a $k$-variety~$X$ is \textit{split} if it contains a geometrically integral open subscheme. If $L$ is a field extension of $k$, then we say that $L$ is \textit{a splitting field }of~$V$ if~$V\otimes_kL$ is split. A splitting field of $V$ is \textit{minimal} if it does not contain any proper subfield that splits $V$.

\subsection{Generalities on $2$-quasi-finite fields}\label{subsec2QF}

Let us start by giving general properties of $2$-quasi-finite fields.

\begin{prop}\label{2QF}
Let $k$ be a $2$-quasi-finite field and $l/k$ a finite extension of $k$.
\begin{enumerate}[label=(\alph*)]
\item The cohomological dimension of $k$ is one.
\item We have $l^{\times}/(l^{\times})^2\simeq\mathbf{Z}/2\mathbf{Z}$.
\item The extension $l/k$ contains a unique extension of degree $2$ of $k$.
\end{enumerate}
\end{prop}
\begin{proof}
Denote by $G$ the absolute Galois group of $k$. First notice that assertions (b) and~(c) are immediate combinations of Galois correspondence and the fact that the $2$-Sylow of~$G$ is~$\mathbf{Z}_2$. For (a), using that $G$ is a closed subgroup of $\widehat{\mathbf{Z}}$, \cite[Chapitre~I,~\S3.3, Proposition~14]{MR1324577} ensures that $\cd(G)\leq\cd(\widehat{\mathbf{Z}})$. This proves that $\cd(G)\leq1$ by \cite[Chapitre~I,~\S3.2, Exemple 1)]{MR1324577}, which is an equality since~$G\neq1$.
\end{proof}

The following proposition supplies a description of corestriction maps for finite extensions of a $2$-quasi-finite field.

\begin{prop}\label{cores2QF}
If the characteristic of $k$ is different from~$2$ and $E/F/k$ are finite field extensions of~$k$, then we have canonical isomorphisms $\coh^1(F,\mathbf{Z}/2\mathbf{Z})\simeq\mathbf{Z}/2\mathbf{Z}$ and $\coh^1(E,\mathbf{Z}/2\mathbf{Z})\simeq\mathbf{Z}/2\mathbf{Z}$ under which the morphism $\cores_{E/F}:\coh^1(E,\mathbf{Z}/2\mathbf{Z})\rightarrow\coh^1(F,\mathbf{Z}/2\mathbf{Z})$ is the identity morphism of~$\mathbf{Z}/2\mathbf{Z}$.
\end{prop}
\begin{proof}
For the first part of the statement, since $k$ has characteristic different from $2$, Kummer's exact sequence ensures that $\coh^1(E,\mathbf{Z}/2\mathbf{Z})\simeq E^{\times}/(E^{\times})^2$ and $\coh^1(F,\mathbf{Z}/2\mathbf{Z})\simeq F^{\times}/(F^{\times})^2$, and they are isomorphic to $\mathbf{Z}/2\mathbf{Z}$ by (b) of Proposition \ref{2QF}. For the second part of the statement, since~$k$ has cohomological dimension one, the map $\cores_{E/F}:\coh^1(E,\mathbf{Z}/2\mathbf{Z})\rightarrow\coh^1(F,\mathbf{Z}/2\mathbf{Z})$ is an isomorphism by \cite[Proposition 3.3.11]{MR2392026}. The only automorphism of $\mathbf{Z}/2\mathbf{Z}$ being the identity map, this proves the statement.
\end{proof}

\subsection{Residues and splitness}\label{subsecCB}


%

\begin{defi}
Let $S$ be a scheme. Two conic bundles $f:X\rightarrow S$ and $g:Y\rightarrow S$ are said to be \textit{equivalent} if there exists a dense open subset $U$ of $S$ such that $f^{-1}(U)$ and $g^{-1}(U)$ are isomorphic over $U$.
\end{defi}

We recall the following proposition, which is useful in the rest of the article.

\begin{prop}[{\cite[Lemma 11.3.2]{MR4304038}}]\label{equweakconicbundle}
Let $C$ be a smooth and geometrically integral curve. If $f:X\rightarrow C$ is a conic bundle, then there exists a regular conic bundle $g:Y\rightarrow C$ equivalent to $f$ such that all fibres of $g$ are integral.
\end{prop}

The following proposition describes a minimal splitting field of the special fibre of a regular conic bundle over a DVR, via the residue of the generic fibre.
\begin{prop}\label{splitresiduemap}
Let $R$ be a discrete valuation ring, $K$ its fraction field,~$\kappa$ its residue field. Assume further that~$2$ is invertible in $R$ and denote by $r$ the associated residue map defined as in~(\ref{residuemap}). Let $C$ be a~smooth~conic over~$K$ and $\mathscr{X}$ an integral $R$-proper scheme such that $\mathscr{X}$ is regular, with generic fibre $\mathscr{X}_K\simeq C$. Let $\alpha\in \kappa^{\times}$ be a representative of the class~$r([C])\in\coh^1(\kappa,\mathbf{Z}/2\mathbf{Z})=\kappa^{\times}/(\kappa^{\times})^2$. Then $\kappa(\sqrt{\alpha})$ is a minimal splitting field of~$\mathscr{X}_{\kappa}$.
\end{prop}
Before proving Proposition \ref{splitresiduemap}, we recall that geometrically, a conic is either isomorphic to a projective line, or a union of two projective lines, or a double line, so that a conic is split if and only if it is a union of two projective lines, or it is geometrically irreducible.
\begin{proof}
We first construct an explicit integral $R$-scheme $\mathscr{C}$ such that $\mathscr{C}$ is regular, with generic fibre $\mathscr{C}_K\simeq C$ and we prove that the statement holds for~$\mathscr{C}$. Then, for $\mathscr{X}$ as in the statement, since $\mathscr{X}_K\simeq\mathscr{C}_K$ the conclusion follows from \cite[Corollary~2.3]{MR3467133}, whose statement is true without assuming that $\kappa$ is perfect, the proof being given in the Remark following it.

To construct $\mathscr{C}$, denote by $v$ the valuation on $R$, pick $\varpi$ a uniformiser of $R$ and choose~$q\in K[x,y,z]$ a quadratic form such that $C=V(q)\subset\mathbf{P}^2_K$. Let us prove that $q$ may be chosen of the form $ax^2+by^2-z^2$ where $a\in R$ and~$b\in R^{\times}$. Since~$\kappa$ is of characteristic different from $2$, we may assume that $q$ is diagonal and write it~$q=ax^2+by^2+cz^2$ where $a,b,c$ further lie in $K^{\times}$ as~$C$ is smooth. After multiplying~$q$ by a constant, we may also assume that $a,b,c\in R$. Furthemore, after the change of variables $x'=\varpi^{\lfloor v(a)/2\rfloor}x$, $y'=\varpi^{\lfloor v(b)/2\rfloor}y$ and $z'=\varpi^{\lfloor v(c)/2\rfloor}z$ we may assume that $v(a),v(b),v(c)\in\{0,1\}$. Since two among the integers $v(a)$, $v(b)$ and $v(c)$ have same class in $\mathbf{Z}/2\mathbf{Z}$, we may rearrange the variables of $q$ in such a way that $v(b)=v(c)$. If the latter is equal to~$0$, we may then divide~$q$ by $-c\in R^{\times}$ so that we may assume that $q=ax^2+by^2-z^2$ with~$v(b)=0$. Otherwise, we may divide~$q$ by $c$ and make the change of variable $x'=\varpi^{\varepsilon}x$ where $\varepsilon=0$ if $v(a)=1$, and $\varepsilon=-1$ is $v(a)=0$, so that we may again assume that $q=ax^2+by^2-z^2$ with $v(b)=0$.

Let us now verify the statement for $\mathscr{C}=V(q)\subset\mathbf{P}^2_R$, which is regular, and for which we clearly have that $\mathscr{C}_K\simeq C$. By \cite[Equation (1.18)]{MR4304038}, we have that $r([C])=[b^{v(a)}]$. If $v(a)=1$, then $a=0\in\kappa$. Thus, $\mathscr{C}_{\kappa}$ is the zero locus of the quadric $by^2-z^2$ which may be rewritten in $\kappa(\sqrt{b})$ as $(y\sqrt{b}-z)(y\sqrt{b}+z)$, so that $\kappa(\sqrt{b})$ splits $\mathscr{C}_{\kappa}$. Moreover, if $b$ is a square in~$\kappa$ then $\kappa(\sqrt{b})$ is clearly minimal. Otherwise, $by^2-z^2$ is irreducible, so that~$\kappa(\sqrt{b})$ is again minimal. Now, if $v(a)=0$, then $\mathscr{C}_{\kappa}$ is the zero locus of a non-degenerate, hence geometrically irreducible conic, so that $\mathscr{C}_{\kappa}$ is split. This proves the statement for $\mathscr{C}$.
\end{proof}

The following proposition is also used thereafter.

\begin{prop}[{\cite[Corollary~2.3 and Remark]{MR3467133}}]\label{NSbirequ}
Let $C$ be a smooth geometrically integral curve over a field $k$ of characteristic different from $2$. If $f:X\rightarrow C$ and $g:Y\rightarrow C$ are two equivalent regular conic bundles, then the non-split fibres of~$f$ and~$g$ lie over the same points of $C$.
\end{prop}

\subsection{Approximation over function fields}\label{rappelapproximation}

Let $C$ be a nice curve over a finite field $\mathbf{F}$, and $K$ its function field, we denote by~$\Omega_K$ the set of places of $K$, that is, closed points of $C$. If $X$ is a proper $K$-variety and $S\subset\Omega_K$ is finite, we set $X(K_{\Omega}^S)=\prod_{v\in\Omega_K\setminus S}X(K_v)$ which is endowed with the product topology. When $S=\emptyset$, this set is also denoted by $X(K_{\Omega})$.

\begin{defi}
Let $X$ be a proper $K$-variety. We say that $X$ has \textit{weak weak approximation} if there exists a finite subset $S\subset\Omega_K$ such that the diagonal embedding $X(K)\xhookrightarrow{} X(K_{\Omega}^S)$ has a dense image. We further say that $X$ has \textit{weak approximation} if we can take $S=\emptyset$.
\end{defi}

The \textit{Brauer-Manin pairing} has been introduced by Manin \cite{MR0427322} to study the defect of weak approximation in the setting of number fields (see also \cite[\S1.3]{MR2439198} and \cite[\S13.3.1]{MR4304038}), although the definition is the same over function fields of curves over a finite field. If $X$ is a proper $K$-variety, it is defined as:
\begin{center}
\begin{tikzcd}[row sep=0.5em, /tikz/column 1/.style={column sep=0em}]
\langle\cdot,\cdot\rangle_{BM} : & X(K_\Omega)\times \br(X) \arrow[r] & \mathbf{Q}/\mathbf{Z} \\
& \left((x_v),\alpha\right) \arrow[r,mapsto] & \sum_{v\in\Omega_K} \inv_v\left(x_v^*(\alpha)\right),
\end{tikzcd}
\end{center}
where $\inv_v:\br(K_v)\rightarrow\mathbf{Q}/\mathbf{Z}$ is the local invariant and $x_v^*$ stands for the specialisation morphism of Brauer groups $\br(x_v):\br(X)\rightarrow\br(K_v)$. Elements of $X(K_{\Omega})$ orthogonal to~$\br(X)$ form a closed subset $X(K_{\Omega})^{\br(X)}$ of $X(K_{\Omega})$ containing $X(K)$ and which is called the Brauer-Manin set of $X$.

\begin{defi}
Let $X$ be a proper $K$-variety. We say that \textit{the Brauer-Manin obstruction to weak approximation is the only one on} $X$ if $X(K)$ is dense in~$X(K_{\Omega})^{\br(X)}$. We abbreviate this by saying that $X$ verifies (BM).
\end{defi}

We need to recall how approximation of local points can be translated over function fields and we refer the reader to \cite[Section~$1$]{MR2931861} for further details.

We fix a proper $K$-variety $X$ and a \textit{model} of $X$, that is, a flat proper morphism $\rho:\mathscr{X}\rightarrow C$ whose generic fibre is isomorphic to $X$. Then, if~$v\in\Omega_K$ and $P_v\in X(K_v)$, the valuative criterion of properness ensures that~$P_v$ extends to a unique $\widehat{\mathscr{O}_v}$-morphism $\widehat{P_v}:\spec(\widehat{\mathscr{O}_v})\rightarrow\mathscr{X}\times_C\spec(\widehat{\mathscr{O}_v})$. If $N$ is a positive integer, we set
$$
J_{P_v,N}\coloneqq\{Q\in X(K_v):\text{ the restrictions of $\widehat{Q}$ and $\widehat{P_v}$ to $\spec(\widehat{\mathscr{O}_v}/\mathfrak{m}_v^N)$
 are the same}\}.
$$

\begin{remi}[{\cite[Section~$1$]{MR2931861}}]\label{remiAPPROX}
If $S\subset\Omega_K$ is finite and $(P_v)_{v\in\Omega_K\setminus S}\in X(K_{\Omega}^S)$, then a fundamental system of neighbourhoods of $(P_v)$ is given by the sets $$W_{T,N}\coloneqq\prod_{v\in T}J_{P_v,N}\times\prod_{v\in\Omega_K\setminus (S\cup T)}X(K_v)$$ where $T$ ranges over finite subsets of $\Omega_K\setminus S$ and~$N$ ranges over positive integers. Furthermore, the mapping $P\mapsto \widehat{P}$ is a bijection between $X(K)\cap W_{T,N}$ and sections $:C\rightarrow \mathscr{X}$ of~$\rho$ such that for all $v\in T$, the pullbacks of $j$ and $\widehat{P_v}$ to $\spec(\widehat{\mathscr{O}_v}/\mathfrak{m}_v^N)$ coincide.
\end{remi}

\section{A unirationality criterion}\label{secCRIT}

This section is dedicated to the proof of the following unirationality criterion. To state it, let us recall that if $\varphi:\mathbf{P}^1_k\rightarrow\mathbf{P}^1_k$ is a dominant morphism and $s,t\in\mathbf{P}^1_k$ are such that~$\varphi(t)=s$, then we denote by $e(t/s)$ the \textit{ramification index} of $\varphi$ at $t$.

\begin{thm}\label{uniratcrit}
Let $k$ be a field of characteristic different from $2$ with $\cd(k)\leq1$, and $f:X\rightarrow\mathbf{P}^1_k$ a regular conic bundle . Denote by~$B$ the set of points in~$\mathbf{P}^1_k$ over which the fibre of~$f$ is non-split. Then the following assertions are equivalent:
\begin{enumerate}[label=\arabic*)]
\item The variety $X$ is $k$-unirational.
\item There exists a dominant morphism $\varphi:\mathbf{P}^1_k\rightarrow\mathbf{P}^1_k$ such that for any $s\in B$ and any~$t\in\varphi^{-1}(s)$ one has $2\mid e(t/s)\times[\kappa(t):\kappa(s)]$.
\end{enumerate}
\end{thm}

\subsection{A criterion of Enriques}\label{subsecENRIQUES}

Before proving Theorem \ref{uniratcrit}, let us state Enriques criterion for the unirationality of conic bundles \cite[Proposition 10.1.1]{MR1668575}.

\begin{prop}\label{enriques}
Let $k$ be a field and $f:X\rightarrow S$ a conic bundle between $k$-varieties. Then the following assertions are equivalent:
\begin{enumerate}[label=(\roman*)]
\item The variety $X$ is unirational.
\item There exists a rational map $g:\mathbf{P}^{\dim S}_k\dashrightarrow X$ such that $f\circ g$ is dominant.
\item There exists a dominant map $h:\mathbf{P}^{\dim S}_k\dashrightarrow S$ such that the base change of $f$ by $h$ has a rational section.
\end{enumerate}
\end{prop}


\begin{proof}[Proof of Proposition \ref{enriques}]
The proof of (ii)$\Rightarrow$(iii)$\Rightarrow$(i) is given in \cite[Proof of Proposition~10.1.1]{MR1668575}.

Since the proof of (i)$\Rightarrow$(ii) given in \textit{ibidem} works implicitly for an infinite field, let us give a general proof. Let us assume (i), that is $X$ is unirational, and let us prove~(ii). First choose a rational dominant map $\psi:\mathbf{P}^r\dashrightarrow X$. Note that the proof of \cite[Lemma~2.3]{MR1956057} ensures that, given a rational map $\varphi:\mathbf{P}^d_k\dashrightarrow X$ with $f\circ \varphi$ dominant, if~$d>\dim S$ then there exists a rational hypersurface $\iota:Z\xhookrightarrow{}\mathbf{P}^d_k$ such that $f\circ\varphi\circ\iota$ is dominant, that is, there exists a rational map $\theta:\mathbf{P}^{d-1}_k\dashrightarrow\mathbf{P}^d_k$ such that~$f\circ\varphi\circ\theta$ is dominant. When starting with $\psi$, we may now apply this  procedure iteratively, which proves~(ii).
%
%
\end{proof}

\begin{rem}\label{remenriques}
Note that if $S$ is a projective curve, then the rational maps in (ii) and~(iii) of Proposition \ref{enriques} may be assumed to be morphisms by the valuative criterion of properness.
\end{rem}

When the base field has cohomological dimension at most one, we infer the following unirationality criterion for conic bundles.

\begin{cor}\label{enriquesP1}
Let $k$ be a field of characteristic different from $2$ such that $\cd(k)\leq1$. If $f:X\rightarrow\mathbf{P}^1_k$ is a regular conic bundle, then $X$ is unirational if and only if there exists a dominant morphism $\varphi:\mathbf{P}^1_k\rightarrow\mathbf{P}^1_k$ such that the closed fibres of all regular conic bundles equivalent to the base change of $f$ by~$\varphi$ are split.
\end{cor}

The proof of Corollary \ref{enriquesP1} requires the following characterisation of regular conic bundles over $\mathbf{P}^1_k$ having a section.

\begin{lem}\label{sectionsplit}
Let $k$ be a field of characteristic different from $2$, with $\br(k)\left[2\right]=0$. If $f:X\rightarrow\mathbf{P}^1_k$ is a regular conic bundle, then~$f$ has a section if and only if all fibres of $f$ over a closed point of $\mathbf{P}^1_k$ are split.
\end{lem}
\begin{proof}
If we write the short exact of \cite[Theorem~6.9.1]{MR2266528} with $i=m=2$ and $j=1$, we get the following short exact sequence
\begin{equation}\label{puritybrauerP1}
\begin{tikzcd}
0\arrow[r] & \br(\mathbf{P}^1_k)[2]\arrow[r] & \br(k(\mathbf{P}^1_k))[2] \arrow[r,"\oplus r_P"] & \displaystyle\bigoplus_{P\in(\mathbf{P}^1_k)^{(1)}}\coh^1(\kappa(P),\mathbf{Z}/2\mathbf{Z})
\end{tikzcd}
\end{equation}
where for each $P\in(\mathbf{P}^1_k)^{(1)}$, the map~$r_P$ denotes the residue map (\ref{residuemap}) associated to the discrete valuation ring~$\mathscr{O}_{\mathbf{P}^1_k,P}$. Besides,~$\br({\mathbf{P}^1_k})\left[2\right]=\br(k)\left[2\right]=0$ where the first equality is derived from \cite[Corollary 2.3.9]{MR1845760} and the second one from our assumption on $k$. Thus, if $X_\eta$ denotes the generic fibre of~$f$, then~$f$ has a section if and only if the smooth conic $X_\eta$ has a rational point, that is, if and only if~$[X_\eta]=0\in\br(k(\mathbf{P}^1_k))[2]$. By the short exact sequence (\ref{puritybrauerP1}), $[X_\eta]=0$ if and only if for any closed point $P$ of $\mathbf{P}^1_k$, $r_P([X_\eta])=0$. Furthermore, Proposition \ref{splitresiduemap} applied to the regular integral scheme $X\times_{\mathbf{P}^1_k}\spec(\mathscr{O}_{\mathbf{P}^1_k,P})$ ensures that $r_P([X_\eta])=0$ if and only if~$X_P$ splits. This proves that $f$ has a section if and only if all its closed fibres are split.
\end{proof}
\begin{rem}\label{rem2torbr}
Let $k$ be a field of characteristic different from $2$. Then $\br(k)\left[2\right]=0$ if $k$ verifies one of the following conditions:
\begin{enumerate}[label=(\alph*)]
\item the cohomological dimension of $k$ is at most one;
\item the absolute Galois group of $k$ is a $p$-group for some prime number $p\neq2$.
\end{enumerate}
Indeed, if $k$ satisfies (a), this comes from the isomorphism $\br(k)\left[2\right]=\coh^2(k,\mathbf{Z}/2\mathbf{Z})$ induced by Kummer's exact sequence. If $k$ satisfies (b), all central simple algebras $A$ over $k$ are split by a finite separable extension of $k$, whose degree is by assumption a power of the odd prime $p$, that is, $[A]\not\in\br(k)\{2\}\setminus\{0\}$.
\end{rem}

We may now give a proof of Corollary \ref{enriquesP1}:
\begin{proof}[Proof of Corollary \ref{enriquesP1}]
By Remark \ref{remenriques} applied to (iii) of Proposition \ref{enriques}, $X$ is unirational if and only if there exists a dominant map $\varphi:\mathbf{P}^1_k\rightarrow\mathbf{P}^1_k$ such that the conic bundle $g:X\times_{\mathbf{P}^1_k,\varphi}\mathbf{P}^1_k\rightarrow\mathbf{P}^1_k$, defined as the base change of $f$ by $\varphi$, has a section. By Proposition~\ref{equweakconicbundle}, all conic bundles are equivalent to a regular conic bundle. In particular, $g$ has a section if and only if all regular conic bundles over $\mathbf{P}^1_k$ equivalent to $g$ have a section. Now, using Lemma~\ref{sectionsplit} and (a) of Remark \ref{rem2torbr}, the latter is equivalent to saying that the closed fibres of all regular conic bundle over $\mathbf{P}^1_k$ equivalent to $g$ are split, which proves the statement.
\end{proof}

\subsection{Proof of Theorem \ref{uniratcrit}}\label{subsecUNIRATCRIT}

Before we give a proof of Theorem \ref{uniratcrit}, we recall that when $R\subset S$ is an inclusion of discrete valuation rings, with respective fraction fields $K\subset L$ and residue fields $\kappa\subset\lambda$ of characteristic different from $2$, we denote by $\Res_{K/L}:\br(K)[2]\rightarrow\br(L)[2]$ the restriction morphism defined contravariantly from the morphisms $\spec(L)\rightarrow\spec(K)$. If $R$ contains a field and $e$ is the ramification index of $S/R$, the following diagram is commutative \cite[Proposition 1.4.7]{MR4304038}:
\begin{equation}\label{residuesquare}
\begin{tikzcd}
\br(K)[2]\arrow[r, "r_K"]\arrow[d, "\Res_{K/L}"] & \coh^1(\kappa,\mathbf{Z}/2\mathbf{Z}) \arrow[d, "e\times\res_{\kappa/\lambda}"] \\
\br(L)[2]\arrow[r, "r_L"] & \coh^1(\lambda,\mathbf{Z}/2\mathbf{Z})
\end{tikzcd}
\end{equation}
where the horizontal maps are the residue maps (\ref{residuemap}) associated to $R$ and $S$.

\begin{proof}[Proof of Theorem \ref{uniratcrit}]
Before starting the proof, using Propositions~\ref{equweakconicbundle} and \ref{NSbirequ}, we may assume that all non-split fibres of $f$ are integral. Thus, non-split fibres of $f$ coincide with its singular fibres.

Let us now make a remark that is useful all along the proof. If $f:X\rightarrow\mathbf{P}^1_k$ is a conic bundle and $\varphi:\mathbf{P}^1_k\rightarrow\mathbf{P}^1_k$ is dominant, denote by $f':X'\rightarrow\mathbf{P}^1_k$ the base change of~$f$ by~$\varphi$ (so that $X'$ may not be regular). For any closed point $t$ of $\mathbf{P}^1_k$, if we set~$s=\varphi(t)$ and if we denote by~$r_t$ (resp. $r_s$) the residue map (\ref{residuemap})\ at $t$ (resp. at $s$) and $X_\eta$ (resp.~$X'_\eta$) the generic fibre of~$f$ (resp.~$f'$) then if we apply the commutative diagram (\ref{residuesquare}) to the inclusion of discrete valuation rings~$\mathscr{O}_{\mathbf{P}^1_k,s}\subset \mathscr{O}_{\mathbf{P}^1_k,t}$ induced by $f$, we get
\begin{equation}\label{protoresiduebasechange}
r_t(X'_\eta)=e(t/s)\times\res_{\kappa(s)/\kappa(t)}(r_s(X_\eta)).
\end{equation}
As $\cd(k)\leq1$, the map $\cores_{\kappa(t)/\kappa(s)}$ is an isomorphism by \cite[Proposition~3.3.11]{MR2392026}. Moreover, $\cores_{\kappa(t)/\kappa(s)}\circ\res_{\kappa(s)/\kappa(t)}=[\kappa(t):\kappa(s)]$ \cite[Proposition 4.2.10]{MR2266528}, so that~(\ref{protoresiduebasechange}) may be rewritten as
\begin{equation}\label{residuebasechange}
r_t(X'_\eta)=e(t/s)\times[\kappa(t):\kappa(s)]\times \cores_{\kappa(t)/\kappa(s)}^{-1}\left(r_s(X_\eta)\right)\in\coh^1(\kappa(t),\mathbf{Z}/2\mathbf{Z}).
\end{equation}
Now, if $f'':X''\rightarrow\mathbf{P}^1_k$ is a regular conic bundle equivalent to $f'$, since its generic fibre $X''_{\eta}$ is isomorphic to $X'_{\eta}$, we may rewrite (\ref{residuebasechange}) as
\begin{equation}\label{regresiduebasechange}
r_t(X''_\eta)=e(t/s)\times[\kappa(t):\kappa(s)]\times \cores_{\kappa(t)/\kappa(s)}^{-1}\left(r_s(X_\eta)\right)\in\coh^1(\kappa(t),\mathbf{Z}/2\mathbf{Z}).
\end{equation}

Let us now assume 1), that is $X$ is $k$-unirational. By Corollary \ref{enriquesP1}, there exists a dominant map $\varphi:\mathbf{P}^1_k\rightarrow\mathbf{P}^1_k$ such that if $f'':X''\rightarrow\mathbf{P}^1_k$ is a regular conic bundle equivalent to the base change $f':X'\rightarrow\mathbf{P}^1_k$ of $f$ by $\varphi$, then all closed fibres of $f''$ are split. For $s\in B$ and $t\in\varphi^{-1}(s)$, since~$X''_t$ splits and~$X_s$ is non-split, Proposition \ref{splitresiduemap} ensures that~$r_t(X''_\eta)=0$ and $r_s(X_\eta)\neq0$. As~$\coh^1(\kappa(t),\mathbf{Z}/2\mathbf{Z})=\kappa(t)^{\times}/(\kappa(t)^{\times})^2$ is a group of exponent~$2$, by (\ref{regresiduebasechange}) we deduce that $2\mid e(t/s)\times[\kappa(t):\kappa(s)]$, which proves 2).

Suppose that 2) holds true and let $f':X'\rightarrow\mathbf{P}^1_k$ be the base change of~$f$ by~$\varphi$ and $f'':X''\rightarrow\mathbf{P}^1_k$ a regular conic bundle such that $f'$ and $f''$ are isomorphic over $\mathbf{P}^1_k\setminus B$. In the beginning of the proof, $X$ has been chosen such that $f$ is smooth above $\mathbf{P}^1_k\setminus B$, so that the morphism $f'$ is also smooth above~$\mathbf{P}^1_k\setminus\varphi^{-1}(B)$. In particular, the fibres of $f''$ above~$U$ are smooth, hence split. Furthermore, for $s\in B$ and $t\in\varphi^{-1}(s)$, the assumption on $\varphi$ and equation (\ref{regresiduebasechange}) ensure that $r_t(X''_{\eta})=0$, as $\coh^1(\kappa(t),\mathbf{Z}/2\mathbf{Z})$ is a group of exponent~$2$. By Proposition \ref{splitresiduemap}, the fibre $X''_t$ is then split. This shows that any closed fibre of $f''$ is split, so that $X$ is unirational by Corollary \ref{enriquesP1}, which proves 1).
\end{proof}

\begin{rem}\label{coverkillresidues}
By Lemma \ref{sectionsplit}, the last paragraph of the proof ensures that if~$\varphi$ verifies the condition of 2), then the base change of $f$ by $\varphi$ has a section.
\end{rem}

%

\subsection{A criterion for triviality of $R$-equivalence}\label{subsecREQUIVCRIT}

\begin{cor}\label{Requicrit}
Let $k$ be a field of characteristic different from $2$ with $\cd(k)\leq1$, and $f:X\rightarrow\mathbf{P}^1_k$ a regular conic bundle. Denote by $B$ the set of points of $\mathbf{P}^1_k$ over which the fibre of $f$ is non-split. Assume that for all $s_0,s_1\in\mathbf{P}^1(k)$ there exists a dominant morphism $\varphi:\mathbf{P}^1_k\rightarrow\mathbf{P}^1_k$ verifying:
\begin{enumerate}[label=(\alph*)]
\item for any $s\in B$ and any $t\in\varphi^{-1}(s)$ one has $2\mid e(t/s)\times[\kappa(t):\kappa(s)]$;
\item the fibres $\varphi^{-1}(s_0)$ and $\varphi^{-1}(s_1)$ have a rational point.
\end{enumerate}
Then $X$ is $k$-unirational and $X(k)/R$ is trivial.
\end{cor}

Let us recall the following elementary observation on conics.
\begin{remi}\label{RequivCONIC}
Let $F$ be a field of characteristic different from $2$ and $C$ a conic over $F$.
\begin{enumerate}[label=(\roman*)]
\item If $C(F)\neq\emptyset$, then $R$-equivalence is trivial on $C$.
\item If $\cd(F)\leq1$, then $C(F)\neq\emptyset$.
\end{enumerate} 
\end{remi}
\begin{proof}
For (i), if $C(F)\neq\emptyset$, then $C$ is either a line, or a union of two distinct lines in~$\mathbf{P}^2_F$ or a double line, so that $R$-equivalence is trivial on it. For (ii), if $C$ is smooth, it has a class in $\br(k)\left[2\right]$. But $\br(k)\left[2\right]=\coh^2(k,\mathbf{Z}/2\mathbf{Z})$, and the latter is zero since $\cd(F)\leq1$. We thus have that $C\simeq\mathbf{P}^1_F$ from which the statement follows. If $C$ is singular, all singular points are rational.
\end{proof}

We may now give a proof of Corollary \ref{Requicrit}.
\begin{proof}[Proof of Corollary \ref{Requicrit}]
By Theorem \ref{uniratcrit}, $X$ is $k$-unirational. Let us prove that~$X(k)/R$ has cardinality one. Since the conic $X_0$ has a rational point, by (ii) of Reminder \ref{RequivCONIC},~$X(k)\neq\emptyset$ so that it remains to prove that for all $x_0,x_1\in X(k)$, the points $x_0$ and $x_1$ are $R$-equivalent.

Set $s_0=f(x_0)$, $s_1=f(x_1)$ which are rational points of~$\mathbf{P}^1_k$. Choose $\varphi:\mathbf{P}^1_k\rightarrow\mathbf{P}^1_k$ as in the statement and denote by $f':X'\rightarrow\mathbf{P}^1_k$ the base change of~$f$ by $\varphi$ and by~$g:X'\rightarrow X$ the base change of~$\varphi$ by $f$. Using (b), let us choose~$t_0$ (resp.~$t_1$) a rational point of~$\varphi^{-1}(s_0)$ (resp. $\varphi^{-1}(s_1)$). Since $\varphi$ verifies (a), Remark \ref{coverkillresidues} ensures that~$f'$ has a section~$h:\mathbf{P}^1_k\rightarrow X'$. Set~$x'_0=h(t_0)$ and~$x'_1=h(t_1)$, so that~$x'_0,x'_1\in X'(k)$. As the rational points~$g(x'_0)$ and~$g(x'_1)$ of~$X$ lie in $(g\circ h)(\mathbf{P}^1(k))$, they are $R$-equivalent. Moreover,~$x_0$ and~$g(x'_0)$ (resp.~$x_1$ and $g(x'_1)$) lie on the same fibre of $f$, which is a conic, hence they are $R$-equivalent by~(i) of Reminder \ref{RequivCONIC}. This proves that $x_0$ and $x_1$ are $R$-equivalent.
\end{proof}

\section{Proof of the main results}\label{secMAINRESU}

In this section, we assume that $k$ is a $2$-quasi-finite field and we prove Theorems \ref{unirationalconicbundles} and \ref{Requivalenceconicbundles}. We respectively make use of Theorem \ref{uniratcrit} and Corollary~\ref{Requicrit}.

\subsection{Some ramified covers of $\mathbf{P}^1_k$}\label{subsecCOVER}

This subsection encapsulates the construction of particular covers of the projective line that are thoroughly used in the proofs of Theorems \ref{unirationalconicbundles} and \ref{Requivalenceconicbundles}.

\begin{lem}\label{reducingdegree}
Let $m$ be a closed point of $\mathbf{P}^1_k$ such that $\deg(m)=2d$, with $d\in\mathbf{Z}_{>0}$. Then there exists a degree $d$ morphism $\varphi:\mathbf{P}^1_k\rightarrow\mathbf{P}^1_k$ such that $\deg(\varphi(m))=2$.
\end{lem}
\begin{proof}[Proof of Lemma \ref{reducingdegree}]
Let $l/k$ be a degree $2$ extension. Note that $l$ is unique up to isomorphism and $\kappa(m)/k$ is Galois, since $k$ is $2$-quasi-finite. All along the proof, we denote by $\sigma$ the nontrivial element of $\gal(l/k)$, we fix $\alpha\in k^{\times}\setminus (k^{\times})^2$ and set~$P\in\mathbf{A}^1_k$ corresponding to the polynomial~$x^2-\alpha$, so that $\deg(P)=2$. For~$h\in l(\mathbf{P}^1)$, we denote by $^{\sigma}h$ the image of $h$ by the left action of $\sigma$ on $l(\mathbf{P}^1)$.

Since $\kappa(m)/k$ is Galois and $l$ is unique up to isomorphism, the extension $l/k$ sits in~$\kappa(m)/k$. Thus, the fibre of $m$ under the morphism $\mathbf{P}^1_l\rightarrow\mathbf{P}^1_k$, which is the fibre product of $\spec(l)\rightarrow\spec(k)$ with $\mathbf{P}^1_k$, is made of two degree $2d$ points $m_1,m_2\in\mathbf{P}^1_l$, as $\kappa(m)\otimes_kl=\kappa(m)\otimes_l(l\otimes_kl)=\kappa(m)\times\kappa(m)$. Since $m_1-m_2$ has degree zero, there exists $f\in l(\mathbf{P}^1)$ such that~$\divcart(f)=m_1-m_2$, so that the map $f:\mathbf{P}^1_l\rightarrow\mathbf{P}^1_l$ it induces verifies~$m_1=f^{-1}(0)$,~$m_2=f^{-1}(\infty)$. Let $g$ be an automorphism of~$\mathbf{P}^1_l$ such that $g(0)=\sqrt{\alpha}$ and~$g(\infty)=-\sqrt{\alpha}$ and~$( ^{\sigma}g)(x)=g(1/\sigma(x))$. We are going to prove that there exists $u\in l^{\times}$ such that~$g\circ(uf)$ is $\sigma$-invariant as an element of~$l(\mathbf{P}^1)$. By Galois descent,~$g\circ(uf)$ will then be the base change of a morphism $\varphi:\mathbf{P}^1_k\rightarrow\mathbf{P}^1_k$. Since $(g\circ(uf))^{-1}\{\sqrt{\alpha},-\sqrt{\alpha}\}=\{m_1,m_2\}$, this means that $\varphi(m)=P$ and~$\varphi$ is of degree $d$, which will prove the statement.

Let us first note that $\divcart( {^{\sigma}f})=m_2-m_1$ so that $\divcart(f\times {^{\sigma}f})=0$, that is, there exists~$v\in l^{\times}$ such that $f\times{^\sigma f}=v$. Furthermore, as $v$ is $\sigma$-invariant, we have $v\in k^{\times}$. Since the cohomological dimension of $k$ is one, the corestriction map $\cores_{l/k}:\coh^1(l,\mathbf{Z}/2\mathbf{Z})\rightarrow\coh^1(k,\mathbf{Z}/2\mathbf{Z})$ is surjective by \cite[Proposition 3.3.11]{MR2392026}. Hence, the norm map $N_{l/k}:l^{\times}\rightarrow k^{\times}$ is surjective. In particular, there exists $u\in l^{\times}$ such that $v=u\times\sigma(u)$. After replacing $f$ by $f/u$, we may then assume that~$f\times{}^{\sigma}f=1$. Thus:
$${}^{\sigma}(g\circ f)=g\circ(1/({}^{\sigma}f))=g\circ f$$
from which we deduce that $g\circ f$ is $\sigma$-invariant.
\end{proof}
\begin{cor}\label{correducingdegree}
If $P$ is a closed point of $\mathbf{P}^1_k$ such that $\deg(P)=2$, then there exists a finite morphism $\varphi:\mathbf{P}^1_k\rightarrow\mathbf{P}^1_k$ of degree $2$ such that $\varphi^{-1}(P)$ is a closed point of degree $4$.
\end{cor}
\begin{proof}
Let $m$ be a point of $\mathbf{P}^1_k$ of degree $4$. By Lemma \ref{reducingdegree}, there exists a finite morphism $\theta:\mathbf{P}^1_k\rightarrow\mathbf{P}^1_k$ of degree $2$ such that $\theta(m)$ is a point of degree $2$. Since $\theta(m)$ and $P$ are both points of degree $2$, there exists an automorphism $\psi$ of $\mathbf{P}^1_k$ mapping $\theta(m)$ to $P$. Thus, the map $\varphi\coloneqq\psi\circ\theta$ is the sought cover.
\end{proof}

\begin{lem}\label{twistingdegreetwo}
Let $\varphi:\mathbf{P}^1_k\rightarrow\mathbf{P}^1_k$ be a degree $2$ cover, $U\subset\mathbf{P}^1_k$ be the complement of the branch locus of $\varphi$. Then there exists a degree $2$ cover $\psi:\mathbf{P}^1_k\rightarrow\mathbf{P}^1_k$ with branch locus~$\mathbf{P}^1_k\setminus U$ such that for any closed point $m\in U$:
\begin{enumerate}[label=(\roman*)]
\item if $\deg(m)$ is odd and $\varphi^{-1}(m)$ is a point of degree $2$ over $m$, then $\psi^{-1}(m)$ is made of two rational points over $m$;
\item if $\deg(m)$ is odd and $\varphi^{-1}(m)$ is made of two rational points over $m$, then~$\psi^{-1}(m)$ is a point of degree $2$ over $m$;
\item if $\deg(m)$ is even, then $\varphi^{-1}(m)$ and $\psi^{-1}(m)$ are isomorphic over $m$.
\end{enumerate}
\end{lem}

\begin{proof}[Proof of Lemma \ref{twistingdegreetwo}]
Let us first note that if $u:U\rightarrow\spec(k)$ is the structural morphism of $U$, then the following diagram is commutative
\begin{center}
\begin{tikzcd}
\mathbf{Z}/2\mathbf{Z} \arrow[rrr,"{\times [\kappa(m):k]}"] \arrow[d,equal] & & & \mathbf{Z}/2\mathbf{Z} \arrow[d,equal] \\
\coh^1(k,\mathbf{Z}/2\mathbf{Z}) \arrow[r,"u^*"]\arrow[rr,bend right, "\res_{\kappa(m)/k}"] & \coh^1(U,\mathbf{Z}/2\mathbf{Z})\arrow[r,"m^*"] & \coh^1(\kappa(m),\mathbf{Z}/2\mathbf{Z}) \arrow[r,"\cores_{\kappa(m)/k}", outer sep=0.5em] & \coh^1(k,\mathbf{Z}/2\mathbf{Z})
\end{tikzcd}
\end{center}
where the vertical equalities come from Proposition \ref{cores2QF}, by $2$-quasi-finiteness of $k$, and the commutativity of the whole diagram is the restriction-corestriction formula \cite[Proposition 4.2.10]{MR2266528}.

Denote by $\alpha$ the nonzero class in $\coh^1(k,\mathbf{Z}/2\mathbf{Z})$ and $\tau:\varphi^{-1}(U)\rightarrow U$ the restriction of~$\varphi$ above $U$, which defines a class in $\coh^1(U,\mathbf{Z}/2\mathbf{Z})$. Let $\tau':V\rightarrow U$ be a $\mathbf{Z}/2\mathbf{Z}$-torsor whose class in $\coh^1(U,\mathbf{Z}/2\mathbf{Z})$ is $u^*(\alpha)+[\tau]$, that is, a twist of~$\tau$ by any element of $k^{\times}\setminus (k^{\times})^2$. Denote by $\psi:Y\rightarrow\mathbf{P}^1_k$ the normalisation of $\mathbf{P}^1_k$ in $\spec(k(V))$, so that $Y$ is a smooth projective curve. Since $Y\otimes_k\overline{k}$ is the normalisation of $\mathbf{P}^1_{\overline{k}}$ in $\overline{k}(V)=\overline{k}(\varphi^{-1}(U))$, this means that~$Y\otimes_k\overline{k}\simeq\mathbf{P}^1_{\overline{k}}$. In particular, $Y$ is a smooth geometrically connected projective curve of genus $0$, that is, $Y$ is isomorphic to a smooth conic. As $\cd(k)=1$, we have~$Y(k)\neq\emptyset$ by Reminder~\ref{RequivCONIC}, so that $Y\simeq\mathbf{P}^1_k$ and $\psi:\mathbf{P}^1_k\rightarrow\mathbf{P}^1_k$.

Let us now prove that $\psi$ is the sought cover. Indeed, in $\coh^1(\kappa(m),\mathbf{Z}/2\mathbf{Z})$, we have:
\begin{equation}\label{fibrepsi}
[\psi^{-1}(m)]=m^*([\tau'])=m^*(u^*(\alpha)+[\tau])=\res_{\kappa(m)/k}(\alpha)+[\varphi^{-1}(m)]\in\mathbf{Z}/2\mathbf{Z}.
\end{equation}
Moreover, $\cores_{\kappa(m)/k}$ is an isomorphism by Proposition \ref{cores2QF}. Besides, from the identity $$\cores_{\kappa(m)/k}\circ\res_{\kappa(m)/k}=[\kappa(m):\kappa]$$
we may rewrite equation (\ref{fibrepsi}) as
$$[\psi^{-1}(m)]=[\kappa(m):k]+[\varphi^{-1}(m)]\in\coh^1(\kappa(m),\mathbf{Z}/2\mathbf{Z})=\mathbf{Z}/2\mathbf{Z}$$
which is the statement we wanted to prove.
\end{proof}

\subsection{Proof of Theorem \ref{unirationalconicbundles}}\label{proofunirationalconicbundles}

Let us prove Theorem \ref{unirationalconicbundles} using Theorem~\ref{uniratcrit}. For this purpose, we show that for any~$B$ as in~(\ref{unirationalcondition}), there exists a cover $\varphi:\mathbf{P}^1_k\rightarrow\mathbf{P}^1_k$ such that for all~$s\in B$ and $t\in\varphi^{-1}(s)$ we have
\begin{equation}\label{evencondition}
2\mid e(t/s)\times[\kappa(t):\kappa(s)].	
\end{equation}

In the following lemma, we start by tackling the case where $B\subset\mathbf{P}^1(k)$.
\begin{lem}\label{killingrationalresidues}
Let $B\subset\mathbf{P}^1(k)$. Then, for any two points $P,Q\in B$, there exists a dominant map $\varphi:\mathbf{P}^1_k\rightarrow\mathbf{P}^1_k$ whose degree is a power of $2$, satisfying condition~(\ref{evencondition}), with the further assumption that $\varphi$ is totally ramified above $P$ and $Q$.
\end{lem}
\begin{proof}
After adding rational points to $B$, we may assume that $|B|$ is even. Setting $|B|=2n$, we prove the statement by induction on $n$. The case where~$n=0$ being trivial, we may assume that $n>0$ and such a $\varphi$ exists for strictly lower~$n$. Fix distinct points~$P$ and $Q$ in~$B$ and, after choosing an automorphism of~$\mathbf{P}^1_k$, assume that~$P=0$ and $Q=\infty$. Then, define $\psi:\mathbf{P}^1_k\rightarrow\mathbf{P}^1_k$ by~$t\mapsto t^2$, so that $\psi$ is totally ramified above $P$ and $Q$. We let~$B_{\mathrm{in}}$ be the set of those~$m\in B\setminus\{P,Q\}$ such that $\psi^{-1}(m)$ is a point of degree $2$ over~$m$, and~$B_{\mathrm{ts}}$ those such that $\psi^{-1}(m)$ is made of two rational points, so that $|B_{\mathrm{in}}|+|B_{\mathrm{ts}}|=2n-2$. By~(i) and~(ii) of Lemma \ref{twistingdegreetwo}, we may assume that $|B_{\mathrm{ts}}|\leq n-1$. Then $\psi^{-1}(B_{\mathrm{ts}})$ is made of at most~$2n-2$ rational points of $\mathbf{P}^1_k$. By induction, there exists $\theta:\mathbf{P}^1_k\rightarrow\mathbf{P}^1_k$ such that~$\theta$ is totally ramified above $\psi^{-1}(P)$,~$\psi^{-1}(Q)$ and condition~(\ref{evencondition}) is verified for $\theta$ and all~$s\in\psi^{-1}(B_{\mathrm{ts}})$ and~$t\in\theta^{-1}(s)$. Thus, if $\varphi\coloneqq\psi\circ\theta$, the cover $\varphi$ is totally ramified above~$P$ and~$Q$. Moreover, for all $s\in B_{\mathrm{ts}}$ and $t\in\varphi^{-1}(s)$, condition (\ref{evencondition}) is verified, and for all~$s\in B_{\mathrm{in}}$ and $t\in\varphi^{-1}(s)$ we have $2\mid[\kappa(t):\kappa(s)]$, so that $\varphi$ is the sought cover.
\end{proof}

We now tackle the general case of Theorem \ref{unirationalconicbundles}.

\begin{proof}[Proof of Theorem \ref{unirationalconicbundles}]
Let us now choose $B$ as in (\ref{unirationalcondition}) and, up to enlarging $B$, we may assume that it is made of rational points, one point $P$ of degree $2$ and one point~$Q$ of odd degree. Corollary~\ref{correducingdegree} supplies a $2$-cover $\psi:\mathbf{P}^1_k\rightarrow\mathbf{P}^1_k$ such that~$\psi^{-1}(P)$ is a point of degree $4$. Using Lemma~\ref{twistingdegreetwo}, we may further assume that $\psi^{-1}(Q)$ is a point of degree $2$ over $Q$. Then, if~$B_{\mathrm{ts}}$ denotes those rational points $m$ of $B\setminus\{P,Q\}$ such that $\psi^{-1}(m)$ is made of two rational points, by $B_{\mathrm{in}}$ those for which~$\psi^{-1}(m)$ is a degree two point over $m$, condition~(\ref{evencondition}) is satisfied for~$s\in\{P,Q\}\cup B_{\mathrm{in}}$ and~$t\in\psi^{-1}(s)$. Furthermore, since $\psi^{-1}(B_{\mathrm{ts}})$ is made of rational points of~$\mathbf{P}^1_k$, Lemma \ref{killingrationalresidues} supplies a morphism $\theta:\mathbf{P}^1_k\rightarrow\mathbf{P}^1_k$ such that condition~(\ref{evencondition}) is verified for~$\theta$ and all~$s\in \psi^{-1}(B_{\mathrm{ts}})$ and~$t\in\theta^{-1}(s)$. By construction, the morphism~$\varphi\coloneqq\psi\circ\theta$ then satisfies condition~(\ref{evencondition}) for all $s\in B$ and $t\in \varphi^{-1}(s)$, so that $\varphi$ is the sought cover.
\end{proof}

\subsection{Proof of Theorem \ref{Requivalenceconicbundles}}\label{proofRequivalenceconicbundles}

Let us now prove Theorem \ref{Requivalenceconicbundles} using Corollary \ref{Requicrit}. If we fix distinct $P,Q\in\mathbf{P}^1(k)$, we then need to find $\varphi:\mathbf{P}^1_k\rightarrow\mathbf{P}^1_k$ such that
\begin{equation}\label{ratpointsinfibrescondition}
\varphi^{-1}(P)(k)\neq\emptyset\text{ and } \varphi^{-1}(Q)(k)\neq\emptyset
\end{equation}
and for any $s\in B$ and $t\in\varphi^{-1}(s)$, condition (\ref{evencondition}) holds.

We start proving the first case of (\ref{Requivalencecondition}), that is, we assume that $B$ is a union of rational points and one point~$m$ of degree $2$. Up to enlarging $B$, we may assume that $P,Q\in B$. Using Corollary \ref{correducingdegree}, there exists a degree $2$ cover $\psi:\mathbf{P}^1_k\rightarrow\mathbf{P}^1_k$ such that $\psi^{-1}(m)$ is a point of degree~$4$. After composing $\psi$ with an automorphism of $\mathbf{P}^1_k$, we may also assume that~$\psi$ is totally ramified above $Q$, and we let $\alpha\coloneqq\psi^{-1}(Q)$. Furthermore, we may assume that~$\psi^{-1}(P)$ contains a rational point. Indeed, if $\psi$ is totally ramified above $P$, then~$\psi^{-1}(P)$ is a rational point, and otherwise, by Lemma~\ref{twistingdegreetwo}, we may assume that~$\psi^{-1}(P)$ is made of two rational points. We thus denote by $\beta$ a rational point of~$\psi^{-1}(P)$. Then, Lemma~\ref{killingrationalresidues} supplies a cover $\theta:\mathbf{P}^1_k\rightarrow\mathbf{P}^1_k$ that is totally ramified above $\alpha$ and $\beta$ and such that condition~(\ref{evencondition}) is satisfied for all $s\in\mathbf{P}^1(k)\cap\psi^{-1}(B\setminus\{P,Q,m\})$ and~$t\in\theta^{-1}(s)$. If we set~$\varphi\coloneqq\psi\circ\theta$, by construction, it verifies condition~(\ref{evencondition}), and it is totally ramified above~$P$ and~$Q$, so that it also satisfies (\ref{ratpointsinfibrescondition}).

We now prove the second case of (\ref{Requivalencecondition}) where $B$ is assumed to be a union of rational points and one point $m$ of odd degree. Again, after enlarging~$B$ and using an automorphism of $\mathbf{P}^1_k$, we may assume that $B$ contains $P$ and $Q$, that~$P=0$ and~$Q=\infty$. We set $\psi:\mathbf{P}^1_k\rightarrow\mathbf{P}^1_k$ the degree $2$ cover defined by $t\mapsto t^2$, which is totally ramified above~$P$ and~$Q$. In the case where $m\not\in\{P,Q\}$, using Lemma~\ref{twistingdegreetwo} we make the additional assumption that $\psi^{-1}(m)$ is a point of degree $2$ over~$m$. Now, Lemma \ref{killingrationalresidues} supplies a cover $\theta:\mathbf{P}^1_k\rightarrow\mathbf{P}^1_k$ that is totally ramified above $\psi^{-1}(P)$ and~$\psi^{-1}(Q)$, and such that condition~(\ref{evencondition}) is satisfied for all $s\in\mathbf{P}^1(k)\cap\psi^{-1}(B\setminus\{P,Q,m\})$ and~$t\in\theta^{-1}(s)$. Then, $\psi\circ\theta$ satisfies condition (\ref{evencondition}) and it is totally ramified above $P$ and $Q$, hence it satisfies condition (\ref{ratpointsinfibrescondition}).

\section{Sufficiency of the Brauer-Manin obstruction}\label{sectionBMcarp}

In this section, when $\mathbf{F}$ is a finite field, we prove that the unirationality of conic bundles over $\mathbf{P}^1_{\mathbf{F}}$ is implied by an analogue of a conjecture of Colliot-Thélène and Sansuc in positive characteristic.

The following conjecture is an analogue of a conjecture of Colliot-Thélène and Sansuc, stated as an open question over number fields in \cite{MR605344} (see also Colliot-Thélène's conjecture in \cite[Conjecture 14.1.2]{MR4304038}).

\begin{conj}\label{BMcarp}
Let $C$ be a nice curve a finite field~$\mathbf{F}$ and $K$ its function field. If $X$ is a proper, smooth, geometrically integral and separably rationally connected surface over~$K$, then~$X$ verifies (BM).
\end{conj}

This section is dedicated to the following theorem on the existence of curves passing through a given set of rational points, for a conic bundle over $\mathbf{P}^1_{\mathbf{F}}$.

\begin{thm}\label{thmKOLLARPOINTS}
Let $\mathbf{F}$ be a finite field, $f:X\rightarrow\mathbf{P}^1_{\mathbf{F}}$ a regular conic bundle. Let $C$ be a nice curve over $\mathbf{F}$ with function field $K$ and assume that $X_K$ verifies (BM). Then the following assertions hold.
\begin{enumerate}[label=(\arabic*)]
\item There exists a morphism $g:C\rightarrow X$ such that $X(\mathbf{F})\subset g(C)$.
\item If the characteristic of $\mathbf{F}$ is odd, then for all~$A\subset X(\mathbf{F})$ with $|A|\leq |C(\mathbf{F})|$, there exists $g:C\rightarrow X$ such that $A\subset g(C(\mathbf{F}))$.
\end{enumerate}
\end{thm}

We split the proof of Theorem \ref{thmKOLLARPOINTS} into two parts. In \S\ref{uniratCONJ} we give a proof of~(1) of Theorem \ref{thmKOLLARPOINTS} and in \S\ref{subsecTRIVREQU}, we show (2) of Theorem \ref{thmKOLLARPOINTS}. Let us now deduce the following immediate corollary on unirationality and triviality of $R$-equivalence for conic bundles.

\begin{cor}\label{fromconj2toconj1}
Let $\mathbf{F}$ be a finite field and $f:X\rightarrow\mathbf{P}^1_{\mathbf{F}}$ a regular conic bundle. Assuming that Conjecture \ref{BMcarp} is true, the following assertions hold.
\begin{enumerate}[label=(\alph*)]
\item The variety $X$ is unirational.
\item If the characteristic of $\mathbf{F}$ is odd, then all two points of $X(\mathbf{F})$ are directly $R$-equivalent.
\end{enumerate}
\end{cor}
\begin{proof}
Let $C\coloneqq\mathbf{P}^1_{\mathbf{F}}$ and $K\coloneqq\mathbf(F)(C)$. Since Conjecture \ref{BMcarp} is true, the variety $X_K$ verifies~(BM).

Let us first prove (a). Since $\mathbf{F}$ is finite, it has cohomological dimension one, so that~$X_0(\mathbf{F})$ and~$X_1(\mathbf{F})$ are nonempty by Reminder \ref{RequivCONIC}. Choose $x\in X_0(\mathbf{F})$ and $y\in X_1(\mathbf{F})$. We then apply (1) of Theorem \ref{thmKOLLARPOINTS} with $C\coloneqq\mathbf{P}^1_{\mathbf{F}}$ to $f$. This supplies a morphism $g: \mathbf{P}^1_F\rightarrow X$ such that $x,y\in g(\mathbf{P}^1_{\mathbf{F}})$. Since $x$ and $y$ lie in distinct fibres of $f$, the morphism~$f\circ g$ is dominant, which implies that~$X$ is unirational by Proposition \ref{enriques}.

Let us now prove (b). If $x,y\in X(\mathbf{F})$ are distinct, then (2) of Theorem \ref{thmKOLLARPOINTS} applied to~$f$ and $C\coloneqq\mathbf{P}^1_{\mathbf{F}}$ supplies $g:\mathbf{P}^1_{\mathbf{F}}\rightarrow X$ such that $x,y\in g(\mathbf{P}^1(\mathbf{F}))$. In other words, $x$ and $y$ are directly $R$-equivalent.
\end{proof}

\subsection{Unirationality}\label{uniratCONJ}

\begin{prop}\label{WWA}
Let $\mathbf{F}$ be a finite field, $C$ a nice curve over $\mathbf{F}$ and $K$ its function field. Let also $X\rightarrow\mathbf{P}^1_{\mathbf{F}}$ be a regular conic bundle. If $X_K$ verifies (BM), then $X_K$ has weak weak approximation.
\end{prop}
Although the proof of Proposition \ref{WWA} relies on classical arguments, it makes use of the finiteness of $\br(X_K)/\br(K)$, for which we supply a proof in Appendix \ref{finiteBRAUER}.
\begin{proof}
First note that $X(\mathbf{F})\neq\emptyset$ since all fibres over a point of $\mathbf{P}^1(\mathbf{F})$ has a rational point by (ii) of Reminder \ref{RequivCONIC}, so that $X(K)\neq\emptyset$. By Corollary \ref{finitebrauerSB}, the group $\br(X_K)/\br(K)$ is finite, and we set $B\subset \br(X_K)$ a finite set of representatives. Then $X(K_{\Omega})^{\br(X)}=\bigcap_{b\in B}X(K_{\Omega})^b$. Since each $X(K_{\Omega})^b$ is open in $X(K_{\Omega})$ by \cite[Corollary 8.2.11.(b)]{MR3729254}, so is $X(K_{\Omega})^{\br(X)}$. As $X(K_{\Omega})^{\br(X)}\neq\emptyset$, this means that there exists a finite $S\subset\Omega_K$ such that the projection map $X(K_{\Omega})^{\br(X)}\rightarrow\prod_{v\in\Omega_K\setminus S}X(K_v)$ is surjective. Since~$X(K)$ is dense in $X(K_{\Omega})^{\br(X)}$, it is also dense in~$\prod_{v\in\Omega_K\setminus S}X(K_v)$.
\end{proof}



Let us now supply a proof of (1) of Theorem \ref{thmKOLLARPOINTS}.

\begin{proof}[Proof of (1) of Theorem \ref{thmKOLLARPOINTS}]
Let us use the notations of the statement. We denote by $\rho:X\times C\rightarrow C$ the projection morphism and we fix $x\in X(\mathbf{F})$. By Proposition \ref{WWA}, the variety $X_K$ has weak weak approximation, so that we can fix a finite set $S\subset \Omega_K$ such that~$X(K)$ is dense in $X(K_{\Omega}^S)$. Let us now choose $T\subset \Omega_K\setminus S$ satisfying $|T|=|X(\mathbf{F})|$, and we fix a bijection $T\rightarrow X(\mathbf{F})$ written as $v\mapsto x_v$. We set $(P_v)\in X(K_{\Omega}^S)$ defined as $P_v\coloneqq x_v$ for $v\in T$ and $P_v\coloneqq x$ for $v\in\Omega_K\setminus(S\cup T)$. Using the notations of Reminder \ref{remiAPPROX}, we then consider the nonempty open subset $W_{T,1}$ of $X(K_{\Omega}^S)$ associated to the tuple $(P_v)$ and the model $\rho$ of $X$. Since $X(K)$ is dense in $X(K_{\Omega}^S)$, we may pick $P\in X(K)\cap W_{T,1}$. Then, using notations of \S\ref{rappelapproximation}, the $K$-point $P$ extends to a section $j:C\rightarrow X\times C$ of $\rho$ that coincides with $\widehat{P_v}$ on $\spec(\kappa(v))$ for $v\in T$. In particular, if $g:C\rightarrow X$ is the first coordinate of~$j$, then for each $v\in T$ we have an equality of set theoretical points $g(v)=x_v$, that is, the sets $g(T)$ and $X(\mathbf{F})$ are the same.
\end{proof}

\subsection{Triviality of $R$-equivalence}\label{subsecTRIVREQU}

\begin{prop}\label{BMzero}
Let $\mathbf{F}$ be a finite field of odd characteristic, $C$ a nice curve over $\mathbf{F}$ and~$K$ its function field. If $f:X\rightarrow\mathbf{P}^1_{\mathbf{F}}$ is a regular conic bundle, then the Brauer-Manin pairing~$\langle\cdot,\cdot\rangle_{BM}$ on $X_K$ is identically zero. In particular, if $X_K$ verifies (BM), then $X_K$ has weak approximation.
\end{prop}

Before we give a proof, let us show how assertion (2) of Theorem \ref{thmKOLLARPOINTS} is inferred from~Proposition \ref{BMzero}.

\begin{proof}[Proof of (2) of Theorem \ref{thmKOLLARPOINTS}]
Let us use the notations of the statement. We denote by $\rho:X\times C\rightarrow C$ the projection morphism and we fix $x\in X(\mathbf{F})$ and $A\subset X(\mathbf{F})$ with $|A|\leq |C(\mathbf{F})|$. By Proposition \ref{BMzero}, the variety $X_K$ has weak approximation, so that~$X(K)$ is dense in $X(K_{\Omega})$. Since $|A|\leq |C(\mathbf{F})|$, let us choose $T\subset C(\mathbf{F})$ satisfying $|T|=|A|$, and we fix a bijection $T\rightarrow A$ written as $v\mapsto x_v$. We set $(P_v)\in X(K_{\Omega})$ defined as $P_v\coloneqq x_v$ for $v\in T$ and $P_v\coloneqq x$ for $v\in\Omega_K\setminus T$. Using the notations of Reminder \ref{remiAPPROX}, we then consider the nonempty open subset $W_{T,1}$ of $X(K_{\Omega}^S)$ associated to the tuple $(P_v)$ and the model~$\rho$ of~$X$. Since $X(K)$ is dense in $X(K_{\Omega})$, we may pick $P\in X(K)\cap W_{T,1}$. Then, using notations of \S\ref{rappelapproximation}, the $K$-point $P$ extends to a section $j:C\rightarrow X\times C$ of $\rho$ that coincides with $\widehat{P_v}$ on $\spec(\kappa(v))=\spec(\mathbf{F})$ for $v\in T$. In particular, if $g:C\rightarrow X$ is the first coordinate of~$j$, then for each $v\in T$, the restriction of $g$ to $\spec(\kappa(v))=\spec(\mathbf{F})$ is~$x_v$. Since $A=\{x_v:v\in T\}$, this shows that $A=g(T)\subset g(C(\mathbf{F}))$.
\end{proof}

The proof of Proposition \ref{BMzero} relies on the following lemma.
\begin{lem}\label{BMchangementdebase}
Let $F$ be a field of characteristic different from $2$ and $C$ a nice curve over $F$ with function field $K$. Let $f:X\rightarrow\mathbf{P}^1_{F}$ be a regular conic bundle with $X(F)\neq\emptyset$. Then the morphism $\br(X\times_F C)\rightarrow\br(X_K)/\br(K)$ is surjective.
\end{lem}
\begin{proof}
To prove the statement, by Lemma \ref{2primaryCB} we need to prove the surjectivity of the map $\br(X\times_F C)\left[2\right]\rightarrow\left(\br(X_K)/\br(K)\right)\left[2\right]$. Since $X(K)\neq\emptyset$ and since the choice of any element of $X(K)$ gives rise to an isomorphism $\br(X_K)\simeq\br(K)\oplus\br(X)/\br(K)$, we have $\left(\br(X_K)/\br(K)\right)\left[2\right]\simeq\br(X)\left[2\right]/\br(F)\left[2\right]$. It is thus enough to show that the map $\br(X\times_F C)\left[2\right]\rightarrow\br(X)\left[2\right]/\br(F)\left[2\right]$ is surjective. 

We denote by $\pi:X\times C\rightarrow C$ the projection morphism, by $\eta$ the generic point of $C$ and $\pi_{\eta}$ the base change of $\pi$ by $\eta$. We then have the following commutative diagram whose rows are exact
\begin{equation}\label{gysindiagram}
\begin{tikzcd}[column sep=1.4em]
0\arrow[r] &\br(X\times C)\left[2\right]\arrow[r] & \br(X_K)\left[2\right]\arrow[r] & \displaystyle\bigoplus_{P\in C^{(1)}} \coh^1_{\et}(X_{\kappa(P)},\mathbf{Z}/2\mathbf{Z})\arrow[r] & \coh^3_{\et}(X\times C, \mathbf{Z}/2\mathbf{Z}) \\
0\arrow[r] &\br(C)\left[2\right]\arrow[r]\arrow[u,"\pi^*"] & \br(K)\arrow[u, "\pi_{\eta}^*"]\left[2\right]\arrow[r] & \displaystyle\bigoplus_{P\in C^{(1)}} \coh^1(\kappa(P),\mathbf{Z}/2\mathbf{Z})\arrow[r]\arrow[u, "\oplus\res_{X_{\kappa(P)}/\kappa(P)}", swap] & \coh^3_{\et}(C, \mathbf{Z}/2\mathbf{Z})\arrow[u, "\pi^*"]
\end{tikzcd}
\end{equation}
where the bottom row (resp. top row) is the limit, as $U$ ranges over nonempty open subsets of $C$, of the exact sequence \cite[(3.17)]{MR4304038} with $Z\coloneqq C\setminus U$ (resp. $Z\coloneqq X\setminus\pi^{-1}(U)$), $l=2$ and $n=1$. The commutativity of the left square is due to functoriality of Brauer groups, that of the central square is due to functoriality of residues (see e.g.\ \cite[Theorem~3.7.5]{MR4304038}) and to the irreducibility of $X_{\kappa(P)}$, and the commutativity of the right square is given by the functoriality of Gysin's spectral sequence \cite[Lemma 2.3.6]{MR4304038}.

Let us prove that $\pi^*:\coh^3_{\et}(C, \mathbf{Z}/2\mathbf{Z})\rightarrow\coh^3_{\et}(X\times C, \mathbf{Z}/2\mathbf{Z})$ is injective and that $\oplus\res_{X_{\kappa(P)}/\kappa(P)}$ is an isomorphism. By diagram chasing on (\ref{gysindiagram}), this will prove that $\br(X\times_F C)\left[2\right]\rightarrow\br(X)\left[2\right]/\br(F)\left[2\right]$ is surjective. Since $X(F)\neq\emptyset$, the morphism~$\pi$ has a section, which proves that $\pi^*$ is injective. Furthermore, for each $P\in C^{(1)}$, we fix~$\overline{\kappa(P)}$ a separable closure of $\kappa(P)$, a geometric point $\overline{x}$ of $X_{\overline{\kappa(P)}}$ and we still denote by $\overline{x}$ the image of $\overline{x}$ in $X_{\kappa(P)}$ and $\spec(\kappa(P))$. Then, \cite[Proposition~5.7.20]{MR2791606} supplies canonical isomorphisms $\coh^1_{\et}(X_{\kappa(P)},\mathbf{Z}/2\mathbf{Z})\simeq\mor_{\cont}(\pi_1(X,\overline{x}),\mathbf{Z}/2\mathbf{Z})$ and $\coh^1(\kappa(P),\mathbf{Z}/2\mathbf{Z})\simeq\mor_{\cont}(\pi_1(\spec(\kappa(P)),\overline{x}),\mathbf{Z}/2\mathbf{Z})$ under which $\res_{X_{\kappa(P)}/\kappa(P)}$ is identified to the pullback $\alpha^*:\mor_{\cont}(\pi_1(\spec(\kappa(P)),\overline{x})\rightarrow\mor_{\cont}(\pi_1(X,\overline{x}),\mathbf{Z}/2\mathbf{Z})$ of the continuous homomorphism $\alpha:\pi_1(X,\overline{x})\rightarrow\pi_1(\spec(\kappa(P)),\overline{x})$, itself induced by the morphism $X_{\spec(\kappa(P))}\rightarrow\spec(\kappa(P))$. Moreover, $\alpha$ is surjective with kernel $\pi_1(X_{\overline{\kappa(P)}},\overline{x})$, see e.g.\ \cite[Proposition~3.3.7]{MR2791606}. But Corollary \ref{finitebrauerSB} ensures that $X_{\overline{\kappa(P)}}$ is rational, which, by purity of the étale fundamental group \cite[X.\S3, Corollaire 3.3]{MR2017446}, implies that $\pi_1(X_{\overline{\kappa(P)}},\overline{x})=1$. In other words, $\alpha$ is an isomorphism, hence so is~$\alpha^*$, that is, $\res_{X_{\kappa(P)}/\kappa(P)}$ is an isomorphism.
\end{proof}

\begin{proof}[Proof of Proposition \ref{BMzero}]
Denote by $h:X\times\spec(K)\rightarrow X\times C$ the product of $id_X$ with the generic point of $C$. Let us fix $v\in \Omega_K$ and $x_v\in X(K_v)$. We write $x_v^*:\br(X_{K_v})\rightarrow\br(K_v)$ (resp. $h^*:\br(X\times C)\rightarrow\br(X\times\spec(K))$) for the pullback of~$x_v$ (resp.~$h$). Let us verify that $x_v^*\circ h^*=0$. For this purpose, from now on, we use the notation of \S\ref{rappelapproximation}. Since the projection morphism $X\times C\rightarrow C$ is a model of $X$, by the valuative criterion of properness, the $K_v$-point $x_v$ extends to a unique $\widehat{\mathscr{O}_v}$-morphism $\widehat{x_v}:\spec(\widehat{\mathscr{O}_v})\rightarrow (X\times C)\times_C\spec(\widehat{\mathscr{O}_v})$. If we still denote by $\widehat{x_v}:\spec(\widehat{\mathscr{O}_v})\rightarrow X\times C$ its projection to~$X\times C$, we thus have a commutative diagram
\begin{center}
\begin{tikzcd}
\spec(K_v)\arrow[d]\arrow[r, "x_v"] & X_K \arrow[d, "h"]\\
 \spec(\mathscr{O}_v)\arrow[r, "\widehat{x_v}"] & X\times C
\end{tikzcd}
\end{center}
so that $x_v^*\circ h^*$ factors through $\widehat{x_v}^*:\br(X\times C)\rightarrow\spec(\mathscr{O}_v)$. But $\br(\mathscr{O}_v)=0$ (see e.g. \cite[Corollary 6.9.3]{MR3729254}) which proves that $x_v^*\circ h^*=0$.

Now, for $\alpha\in \br(X_K)$, Lemma \ref{BMchangementdebase} supplies $\beta\in\im(h^*)$ and $\gamma\in \br(K)$ such that $\alpha=\beta+\gamma$. If $(x_v)\in X(K_{\Omega})$, we thus have $\langle(x_v),\alpha\rangle_{BM}=\langle(x_v),\beta\rangle_{BM}+\langle (x_v),\gamma\rangle_{BM}$. But $\langle(x_v),\beta\rangle_{BM}=\sum_{v\in\Omega_K}\inv_v (x_v^*(\beta))=0$ since $\im(h^*)\subset\ker(x_v^*)$ for all $v\in\Omega_K$, and $\langle (x_v),\gamma\rangle_{BM}=0$ by Albert-Brauer-Hasse-Noether exact sequence. Thus, $\langle(x_v),\alpha\rangle_{BM}=0$, so that the Brauer-Manin pairing on $X_K$ is identically zero.
\end{proof}

\begin{appendices}

\section{Brauer groups of surfaces in positive characteristic}\label{appendixA}

\subsection{Geometrically separably rationally connected surfaces}\label{finiteBRAUER}

Needed in our proof of (1) of Theorem \ref{thmKOLLARPOINTS} is the following statement on Brauer groups of separably rationally connected varieties. We recall that if $X$ is a scheme over a field $F$ and~$F^{\sep}$ is a separable closure of~$F$, we denote the \textit{algebraic Brauer group} of~$X$ by~$\br_1(X)\coloneqq\ker\left[\br(X)\xrightarrow{\mathrm{pr}^*}\br(X\otimes_FF^{\sep})\right]$, where $\mathrm{pr}:X\otimes_FF^{\sep}\rightarrow X$ is the first projection morphism.

\begin{prop}\label{brauerfini}
Let $F$ be a field and $F^{\sep}$ a separable closure of $F$. Let also~$X$ be a smooth, proper and geometrically integral $F$-variety and set $X^{\sep}\coloneqq X\otimes_FF^{\sep}$. If X is separably rationally connected, then~$\br_1(X)/\br(F)$ is finite. In particular, if~$X^{\sep}$ is rational, then~$\br(X)/\br(F)$ is finite.
\end{prop}
\begin{proof}
The statement is well known if the characteristic of $F$ is zero (see e.g.\ \cite[assertion~(6)~in~p.347]{MR4304038}), so we may assume that $F$ has positive characteristic $p$. The second part of the statement is inferred from the first part. Indeed, if $X^{\sep}$ is rational, then $\br(X^{\sep})\simeq\br(\mathbf{P}^{\dim(X)}_{F^{\sep}})$ by purity of the Brauer group \cite{MR3959863}, so that $\br(X^{\sep})=\br(F^{\sep})=0$, hence $\br(X)=\br_1(X)$. Let us then assume that $X$ is separably rationally connected and prove the first part of the statement. We denote by $\overline{F}$ an algebraic closure of $F$ and we set $\overline{X}\coloneqq X\otimes_F\overline{F}$.

We extract from exact sequence (2.23) of \cite[Corollary 2.3.9]{MR1845760} a short exact sequence:
\begin{center}
\begin{tikzcd}
\br(F)\arrow[r, "\pi^*"] & \br_1(X) \arrow[r] & \coh^1(F, \pic(X^s))
\end{tikzcd}
\end{center}
where $\pi:X\rightarrow\spec(F)$ is the structural morphism of $X$. It is thus enough to prove that $\pic(X^{\sep})$ is finitely generated and torsion-free, from which we infer the finiteness of~$\coh^1(F, \pic(X^{\sep}))$, hence that of $\br_1(X)/\br(F)$. First, separably rational connectedness of $X$ ensures that $\coh^1(\overline{X},\mathscr{O}_{\overline{X}})=0$ by (see \cite{MR3158990} and \cite{MR3268754}). This implies that~$\pic(X^{\sep})=\pic(\overline{X})=\ns(\overline{X})$, the Néron-Severi group of $\overline{X}$, which is finitely generated (see e.g.\ \cite[Corollary~5.1.3.(i)]{MR4304038}). Now, using Kummer's exact sequence, for each prime number~$\ell\neq p$ and $m\geq0$ we have $\pic(\overline{X})[\ell^m]=\coh^1_{\et}(\overline{X},\mathbf{Z}/\ell^m\mathbf{Z})$. But since~$\overline{X}$ is separably rationally connected, it is simply connected (see \cite[Theorem~2.1]{MR2538008}), so that~$\coh^1_{\et}(\overline{X},\mathbf{Z}/\ell^m\mathbf{Z})=\pic(\overline{X})[\ell^m]=0$. In particular, $\pic(\overline{X})\{p'\}=0$. Furthermore,~$\pic(\overline{X})\{p\}=0$ by \cite[Theorem 1.4]{MR3831010}. From this we deduce that~$\pic(\overline{X})$ is torsion-free.
\end{proof}

\begin{cor}\label{finitebrauerSB}
Let $F$ be a field and $f:X\rightarrow\mathbf{P}^1_F$ a conic bundle. Then, for all field extension $K$ of $F$ and separable closure $K^{\sep}$ of $K$, the variety $X_{K^{\sep}}$ is rational. In particular, the group~$\br(X_K)/\br(K)$ is finite.
\end{cor}
\begin{proof}
Let $K$ be a field extension of $F$ and $\eta$ the generic point of $\mathbf{P}^1_F$. We let $K^{\sep}$ (resp.~$F^{\sep}$) be a separable closure of $K$ (resp. the separable closure of $F$ contained in $K^{\sep}$) and $\overline{F}$ an algebraic closure of $F$ containing $F^{\sep}$.

Let us first prove that $X_{\overline{F}}$ is rational. Indeed, since~$X_{\eta}$ is geometrically integral we have an isomorphism $\overline{F}(X)\simeq\overline{F}(\mathbf{P}^1)(X_{\eta})$ over $\overline{F}$. But $\br(\overline{F}(\mathbf{P}^1))=0$ by Tsen's theorem so that the class of the smooth conic $X_{\eta}\otimes_{F(\mathbf{P}^1)}\overline{F}(\mathbf{P}^1)$ in $\br(\overline{F}(\mathbf{P}^1))$ is trivial, that is, there exists an isomorphism $X_{\eta}\otimes_{F(\mathbf{P}^1)}\overline{F}(\mathbf{P}^1)\simeq\mathbf{P}^1_{\overline{F}(\mathbf{P}^1)}$ of $\overline{F}(\mathbf{P}^1)$-varieties. This proves that the field $\overline{F}(\mathbf{P}^1)(X_{\eta})$ is purely transcendental over~$\overline{F}(\mathbf{P}^1)$, hence $\overline{F}(X)\simeq\overline{F}(\mathbf{P}^1)(X_{\eta})$ is purely transcendental over $\overline{F}$. In other words, $X_{\overline{F}}$ is rational.

By \cite[Theorem 1]{MR948785}, this implies that $X_{F^{\sep}}$ is also rational, so that the $K^{\sep}$-variety $X_{K^{\sep}}=X_{F^{\sep}}\otimes_{F^{\sep}}K^{\sep}$ is rational. The last part of the statement is a consequence of the last assertion of Proposition \ref{brauerfini}.
\end{proof}

\subsection{Conic bundles over a curve}\label{BrauerCB}

In characteristic zero, it is well known that up to constant classes, the Brauer group of a conic bundle over~$\mathbf{P}^1$ is a $2$-torsion group, see e.g.\ \cite[Corollary 11.3.5]{MR4304038}. In the following proposition, we show that it still holds in odd characteristic.

\begin{prop}\label{2primaryCB}
Let $F$ be a field of characteristic different from $2$ and $X\rightarrow\mathbf{P}^1_F$ a regular conic bundle. If $X(F)\neq\emptyset$, then $\br(X)/\br(F)$ is a $2$-torsion group.
\end{prop}
\begin{proof}[Proof of Lemma \ref{2primaryCB}]
Let us first prove that $\left(\br(X)/\br(F)\right)\{\ell\}=0$ for all prime numbers $\ell\neq2$. For this purpose, we fix a prime number $\ell\neq2$, we denote by $F^{\sep}$ a separable closure of $F$ and we set~$E$ the fixed field of an $\ell$-Sylow of $\gal(F^{\sep}/F)$, so that $\gal(F^{\sep}/E)$ is a pro-$\ell$ group. Using the commutativity of $\br$ with limits \cite[\S2.2.2]{MR4304038} and \cite[Proposition~3.8.4]{MR4304038}, it is enough to prove that $\left(\br(X_E)/\br(E)\right)\{\ell\}=0$. Since $\gal(K^{\sep}/E)$ is a pro-$\ell$ group, $E$ has no quadratic extension, so that all the closed fibres of the conic bundle $f_E\coloneqq f\otimes_FE:X_E\rightarrow\mathbf{P}^1_E$ are split. By Lemma \ref{sectionsplit}, and (b) of Remark \ref{rem2torbr}, there exists a section $s$ of $f_E$. If we denote by $\eta$ the generic fibre of $\mathbf{P}^1_E$, we thus have a commutative diagram
\begin{equation}\label{diag}
\begin{tikzcd}
0\arrow[r] & \br(X_E) \arrow[r]\arrow[d,"s^*", bend left] &\br(X_{E,\eta})\arrow[d,"s_{\eta}^*", bend left] \\
0\arrow[r] & \br(\mathbf{P}^1_E)\arrow[r]\arrow[u,"f^*"] & \br(\eta)\arrow[u,"f_{\eta}^*"]
\end{tikzcd}
\end{equation}
whose rows are exact since $X_E$ is regular, and where $f_{\eta}$ (resp. $s_{\eta}$) is the base change of~$f$ by $\eta$ (resp. is the rational point of $X_{E,\eta}$ corresponding to $s$). Since the conic $X_{E,\eta}$ has a rational point $s_{\eta}$, it is isomorphic to $\mathbf{P}^1_{E(\mathbf{P}^1)}$, so that the map $f_{\eta}^*$ is an isomorphism. As~$s_{\eta}^*\circ f_{\eta}^*=id$, the map $s_{\eta}^*$ is also an isomorphism. In particular, the commutativity of (\ref{diag}) ensures that $s^*$ is injective. But since $s^*\circ f^*=id$, the morphism $s^*$ is also surjective, so that $s^*$ is an isomorphism. This proves that $\br(X_E)\simeq\br(\mathbf{P}^1_E)$, that is, $\br(X_E)=\br(E)$, so that $\left(\br(X_E)/\br(E)\right)\{\ell\}=0$.

It remains to show that $\left(\br(X)/\br(F)\right)\{2\}=\left(\br(X)/\br(F)\right)\left[2\right]$. The choice of an element in $X(F)$ supplies an isomorphism $\br(X)\simeq\br(F)\oplus\br(X)/\br(F)$, so that $\left(\br(X)/\br(F)\right)\{2\}\simeq\br(X)\{2\}/\br(F)\{2\}$. Furthermore, if we denote by $\eta$ the generic point of $\mathbf{P}^1_F$ and $f_{\eta}$ the pullback of $f$ by $\eta$, there is a commutative diagram with exact rows
\begin{center}
\begin{tikzcd}
0\arrow[r] & \br(X) \{2\}\arrow[r] & \br(X_{\eta})\{2\}\arrow[r] & \displaystyle\bigoplus_{P\in (\mathbf{P}^1_F)^{(1)}}\displaystyle\bigoplus_{V\subset X_P}\coh^1(F(V),\mathbf{Q}_2/\mathbf{Z}_2) \\
0\arrow[r] & \br(F) \{2\}\arrow[r]\arrow[u, "f^*"] & \br(F(\mathbf{P}^1))\{2\}\arrow[r]\arrow[u, "f_{\eta}^*"] & \displaystyle\bigoplus_{P\in (\mathbf{P}^1_F)^{(1)}}\coh^1(\kappa(P),\mathbf{Q}_2/\mathbf{Z}_2)\arrow[u, "\oplus_P\oplus_V\res_{F(V)/\kappa(P)}", swap]
\end{tikzcd}.
\end{center}
where the bottom row is the limit, as $U$ ranges over nonempty subsets of $\mathbf{P}^1_F$, of the exact sequences \cite[Theorem~3.7.2.(ii)]{MR4304038} with $Z\coloneqq \mathbf{P}^1_F\setminus U$ and $\ell=2$, combined with the isomorphism $\br(\mathbf{P}^1_F)\simeq\br(F)$. As for the top row,~$V\subset X_P$ ranges over the irreducible components of $X_P$ which, by flatness of $f$, are precisely the codimension one subschemes of $X$. The top row complex is then obtained as the limit, as $U$ ranges over nonempty subsets of $\mathbf{P}^1_F$ of the short exact sequences \cite[Theorem~3.7.2.(3.19)]{MR4304038} applied to $Z\coloneqq X\setminus f^{-1}(U)$ and $\ell=2$. The whole diagram is commutative by the functoriality of residues (see e.g.\ \cite[Theorem 3.7.5]{MR4304038}). Let us notice that since~$X_{\eta}$ is a smooth conic, the middle vertical row is surjective by \cite[Proposition 7.2.1]{MR4304038}. Also, if~$\kappa_V$ denotes the algebraic closure of $\kappa(P)$ in $F(V)$, the kernel of the right vertical map is $\bigoplus_{P\in(\mathbf{P}^1_F)^{(1)}}\bigcap_{V\subset X_P}\coh^1(\gal(\kappa_V/\kappa(P)),\mathbf{Q}_2/\mathbf{Z}_2)$. Since~$X_P$ is a conic, this group is $2$-torsion as $\kappa_V/\kappa(P)$ is an extension of degree at most $2$. By diagram chasing, we deduce that~$2\left(\br(X)\{2\}\right)\subset\br(F)\{2\}$, that is, $\br(X)\{2\}/\br(F)\{2\}$ is a $2$-torsion group, hence~$\left(\br(X)/\br(F)\right)\{2\}$ is $2$-torsion.
\end{proof}


\end{appendices}

\bibliographystyle{myamsalpha}
\bibliography{biblio}
\nocite{*}

\end{document}